\documentclass[12pt]{article}
\usepackage{fullpage,amsthm,amssymb,amsmath}
\usepackage{graphicx,algorithmic,algorithm,amsfonts, verbatim}
\usepackage{enumerate,kotex}
\usepackage{color}
\usepackage{authblk}
\usepackage[pagewise]{lineno}
\newtheorem*{thmk4m}{Theorem~\ref{thm:K4minor}}
\newtheorem*{thmc3}{Theorem~\ref{thm-c3}}
\newtheorem{theorem}{Theorem}[section]
\newtheorem{lemma}[theorem]{Lemma}
\newtheorem{proposition}[theorem]{Proposition}
\newtheorem{question}[theorem]{Question}

\newtheorem{strategy}[theorem]{Strategy}

\newcommand\abs[1]{\lvert #1\rvert}

\begin{document}
 \title{Online Ramsey theory for a triangle on $F$-free graphs}
 
\author[$\ast\mathparagraph$]{Hojin Choi}
\author[$\dagger\mathsection$]{Ilkyoo Choi}
\author[$\ast\mathparagraph$]{Jisu Jeong}
\author[$\ast\parallel \mathparagraph  $]{Sang-il Oum}
\affil[$\mathsection$]{Department of Mathematics, Hankuk University of Foreign Studies, Yongin-si, Gyeonggi-do, Republic of Korea}
\affil[$\parallel$]{Discrete Mathematics Group, Institute for Basic Science (IBS), Daejeon, Republic of Korea}
\affil[$\mathparagraph$]{Department of Mathematical Sciences, KAIST, Daejeon, Republic of Korea}
\footnotetext[1]{Supported by the National Research Foundation of Korea (NRF) grant funded by the Korea government (MSIT) (No. NRF-2017R1A2B4005020).}
\footnotetext[2]{Corresponding author. 
Supported by the Basic Science Research Program through the National Research Foundation of Korea (NRF) funded by the Ministry of Education (NRF-2018R1D1A1B07043049), and also by Hankuk University of Foreign Studies Research Fund.
}
\footnotetext[3]{Emails: \texttt{hojinchoi@kaist.ac.kr},
  \texttt{ilkyoo@hufs.ac.kr}, \texttt{jjisu@kaist.ac.kr},
  \texttt{sangil@kaist.edu}}

\date\today
\maketitle

\begin{abstract}
Given a class $\mathcal{C}$ of graphs and a fixed graph $H$, the \emph {online Ramsey game for $H$ on $\mathcal C$} is a game between two players Builder and Painter as follows: 
an unbounded set of vertices is given as an initial state, and on each turn Builder introduces a new edge with the constraint that the resulting graph must be in $\mathcal C$, and Painter colors the new edge either red or blue. 
Builder wins the game if Painter is forced to make a monochromatic copy of $H$ at some point in the game.
Otherwise, Painter can avoid creating a monochromatic copy of $H$ forever, and we say Painter wins the game.

We initiate the study of characterizing the graphs $F$ such that for a given graph $H$, Painter wins the online Ramsey game for $H$ on $F$-free graphs. 
We characterize all graphs $F$ such that Painter wins the online Ramsey game for $C_3$ on the class of $F$-free graphs, except when $F$ is one particular graph.
We also show that Painter wins the online Ramsey game for $C_3$ on the class of $K_4$-minor-free graphs, extending a result by Grytczuk, Ha\l uszczak, and Kierstead.

%The investigation of online (size) Ramsey theory on specific graph classes was initiated in 2004 by Grytczuk, Ha\l uszczak, and Kierstead.
%They studied online Ramsey theory on forests, $k$-colorable graphs, outerplanar graphs, and planar graphs. 
%Recently, Pet\v{r}\'i\v{c}kov\'a continued the study of online Ramsey theory on planar graphs further by disproving a conjecture by Grytczuk, Ha\l uszczak, and Kierstead.
%
%We carry on the study of online Ramsey theory by considering classes of graphs with one forbidden subgraph. 
%We focus on characterizing the classes where Painter wins the online Ramsey game for $C_3$ on the class of $F$-free graphs; we succeed the characterization except when $F$ is one particular graph.
%We also show that Painter wins the online Ramsey game for $C_3$ on the class of $K_4$-minor-free graphs, extending a result by Grytczuk, Ha\l uszczak, and Kierstead.
%
%We continue the study of online Ramsey theory.
%
%Section 2: We prove that Builder cannot force a monochromatic copy of $C_3$ on $K_4$-minor-free graphs in online Ramsey game, extending a previous result. 
%
%Section 3: We characterize the class of $H$-free graphs where Builder cannot force a monochromatic copy of $C_3$. 
%
%Section 4: We characterize the class of $H$-free graphs where Builder cannot force a monochromatic copy of $F$ for some other small graphs $F$. 
%We also make some conjectures and questions. 
\end{abstract}

\section{Introduction}

All graphs in this paper are finite.
For a host graph $G$ and a target graph $H$, let $G\rightarrow H$ mean that there exists a monochromatic copy of $H$ for every (not necessarily proper) $2$-edge-coloring of $G$.
For a graph parameter $\rho$, let $R_\rho(H)$ denote the minimum $\rho(G)$ where $G\rightarrow H$.
When $\rho$ counts the number of vertices in a graph, $R_\rho(H)$ is \emph {the Ramsey number} of $H$ and it is often denoted $R(H)$. 
The well-known Ramsey's Theorem~\cite{1930Ra} from 1930 states that $R(H)$ is finite for every graph $H$.

Burr, Erd\H os, and Lov\'asz~\cite{1976BuErLo} introduced \emph {the chromatic Ramsey number} and \emph {the degree Ramsey number}, which arises when $\rho$ is the chromatic number and the maximum degree, respectively. 
Erd\H os et al.~\cite{1978ErFaRoSc} introduced \emph {the size Ramsey number}, denoted $R_e(H)$, which arises when $e(G)$ is the number of edges in a graph $G$.
We redirect the readers to a thorough survey by Conlon, Fox, and Sudakov~\cite{2015CoFoSu} for more history regarding these parameters.

Another variant of Ramsey theory is online Ramsey theory, introduced by Beck~\cite{1993Be} in 1993. 
Given a class $\mathcal{C}$ of graphs and a fixed graph $H$, an \emph {online Ramsey game for $H$ on $\mathcal C$} is a game between two players Builder and Painter with the following rules: 
an unbounded set of vertices is given as an initial state, and on each turn Builder introduces a new edge with the constraint that the resulting graph must be in $\mathcal C$, and Painter colors the new edge either red or blue. 
%one vertex is given as an initial state, and on each turn Builder draws finitely many vertices and a new edge with the constraint that the resulting graph must be in $\mathcal C$, and Painter colors the new edge either red or blue. 
Builder wins if Painter is forced to make a monochromatic copy of $H$ at some point of the game, and we say Builder wins the online Ramsey game for $H$ on $\mathcal C$. 
Otherwise, Painter can avoid creating a monochromatic copy of $H$ forever, and we say Painter wins the online Ramsey game for $H$ on $\mathcal C$.
%we say $F$ is \emph {unavoidable} on $\mathcal C$.
%we say $F$ is \emph {avoidable} on $\mathcal C$.

If no graph in $\mathcal{C}$ contains $H$ as a subgraph, then Painter wins the online Ramsey game for $H$ on $\mathcal{C}$ since a copy of $H$ cannot be created, let alone a monochromatic one. 
Therefore it must be that $H$ is a subgraph of at least one graph in $\mathcal C$ for a result to be nontrivial.
If $\mathcal C$ is the class of graphs with bounded maximum degree, then this is the online version of the degree Ramsey number; see~\cite{2011BuGrKiMiStWe,2011Ro,2013Ro} for results regarding this topic. 

This paper concerns the online version of the size Ramsey number.
For a graph $H$, \emph {the online (size) Ramsey number} of $H$, denoted $r(H)$, is the minimum number of rounds required for Builder to win, assuming that both Builder and Painter play optimally. 
When there are no restrictions on the graphs Builder can create, it is an easy consequence of Ramsey's theorem~\cite{1930Ra} that Builder always wins the online Ramsey game for every target graph $H$, so $r(H)\leq R_e(H)$. 
For a fixed graph $H$, studying the ratio of $r(H)$ and $R_e(H)$ was initiated in~\cite{1993Be,2003FrKoRoRuTe,2005KuRu} and has drawn much attention since then~\cite{2004GrHaKi,2008GrKiPr,2009KiKo,2008Pr}.
There is also a line of research trying to determine some exact online Ramsey numbers~\cite{2009Co,2014CyDz,2015CyDzLaLo,2008GrKiPr,2008Pr,2012Pr}.
%,2008 gr ki pr: more for paths 
%2012 pr: onlin paths
%2009 co: online ramsey for complete+complete bipartite
%2014 cy to: online ramsey for c4
%2015 cy dz la lo: online ramsey for paths and cycles
%1978 ne ro 1978 erdos et al 1993 er ro : upper bound online complete bipartite
Additionally, there are some results on the behavior of $r(H)$ in various random settings~\cite{2010BaBu,2007MaMiSt,2009MaSpSt_1,2009MaSpSt,2009PrSpTh}.
%2007 ma mi st
%2009ma sp st1: random
%2009ma sp st: random
%2009 pr sp th: nice survey
%2010 ba bu : random

The investigation of online (size) Ramsey theory on specific graph classes was initiated in 2004 by Grytczuk, Ha\l uszczak, and Kierstead~\cite{2004GrHaKi}.
They studied online Ramsey theory on forests, $k$-colorable graphs, outerplanar graphs, and planar graphs. 
In particular, they conjectured that Builder wins the online Ramsey game for $H$ on planar graphs if and only if $H$ is an outerplanar graph. 
This conjecture was recently disproved by Pet\v{r}\'i\v{c}kov\'a~\cite{2014Pe}; she showed that one direction of the conjecture is true while the other direction is not. 

\begin{proposition}[\cite{2014Pe}]
For every outerplanar graph $H$, Builder wins the online Ramsey game for $H$ on planar graphs. 
%Builder can force a monochromatic copy of every outerplanar graph on planar graphs. 
\end{proposition}

\begin{proposition}[\cite{2014Pe}]
Builder wins the online Ramsey game for $K_{2,3}$ on planar graphs.
%Builder can force a monochromatic copy of $K_{2,3}$ on planar graphs.
\end{proposition}

In~\cite{2004GrHaKi}, it is shown that Painter wins the online Ramsey game for $C_3$ on outerplanar graphs, and the graphs containing $C_3$ as a subgraph are the only known graphs where Painter wins the online Ramsey game on outerplanar graphs.
On the other hand, they also demonstrate that Builder wins the online Ramsey game for $C_3$ on $2$-degenerate planar graphs. 
%; a \emph{$k$-degenerate graph} is a graph where every subgraph has a vertex of degree at most $k$. 

\begin{theorem}[\cite{2004GrHaKi}]\label{thm:outerplanar}
Painter wins the online Ramsey game for $C_3$ on outerplanar graphs.
%Builder cannot force a monochromatic copy of $C_3$ on outerplanar graphs.
\end{theorem}

\begin{proposition}[\cite{2004GrHaKi}]
Builder wins the online Ramsey game for $C_3$ on $2$-degenerate planar graphs. 
%Builder can force a monochromatic copy of $C_3$ on planar $2$-degenerate graphs.
\end{proposition}

%A \emph {minor} of a graph $G$ is a graph obtainable from $G$ by deleting edges and contracting edges. 
%If $H$ is a minor of $G$, then $G$ has a subgraph $F$ contractible to $H$, then we say $F$ is an \emph{$H$-minor} of $G$, and say that $G$ has an $H$-minor. 
%
%If every graph of a class $C$ does not have an $H$-minor, we say $C$ is \emph{$H$-minor-free}.
%In the sense of minor, the class of $K_4$-minor-free graphs strictly contains the class of outerplanar graphs. Also, the class of planar $2$-degenerate graphs strictly contains the class of outerplanar graphs.%(why?)
%\[\{\mbox{outerplanar graphs}\} \subsetneq \{K_4\mbox{-minor-free graphs}\} \subsetneq \{\mbox{planar $2$-degenerate graphs}\}\]

We extend the class of graphs where Painter wins the online Ramsey game for $C_3$ from outerplanar graphs to $K_4$-minor-free graphs. 
Our proof is a generalization of the proof of Theorem~\ref{thm:outerplanar} in~\cite{2004GrHaKi}.
%Builder cannot force a monochromatic copy of $C_3$ on $K_4$-minor free graphs. 

\begin{thmk4m}\label{theorem-2nd}
Painter wins the online Ramsey game for $C_3$ on $K_4$-minor-free graphs.
%Builder cannot force a monochromatic copy of $C_3$ on $K_4$-minor-free graphs.
\end{thmk4m}

We initiate the study of characterizing the graphs $F$ such that for a given graph $H$, Painter wins the online Ramsey game for $H$ on $F$-free graphs. 
%A graph class is \emph {$\{F_1, \ldots, F_k\}$-free} if every graph in the class does not contain any of the graphs $F_1, \ldots, F_k$ as a subgraph. 
A graph class is \emph{$F$-free} if every graph in the class does not contain $F$ as a subgraph. 
%If $k=1$, then we write simply $F_1$-free graphs instead of $\{F_1\}$-free graphs. 
%We characterize the class of $F$-free graphs when Painter wins the online Ramsey game for $C_3$ on $F$-free graphs, except when $F$ is one special graph. 
We characterize all graphs $F$ such that Painter wins the online Ramsey game for $C_3$ on $F$-free graphs, except when $F$ is one special graph. 
%cannot force a monochromatic copy of $C_3$ on. 
%Note that 
We put the constraint that $F$ has no isolated vertices because the game is defined to have infinitely many isolated vertices as the initial state. 
The following theorem is our main result. 

\begin{thmc3}\label{thm-c3}
Let $X_1, \ldots, X_5$ be the graphs in Figure~\ref{fig-c3}, and let $F$ be a graph with no isolated vertices.
Given that $F$ is not isomorphic to $X_5$, Painter wins the online Ramsey game for $C_3$ on $F$-free graphs if and only if $F$ is isomorphic to a subgraph of a graph in $\{X_1, X_2, X_3, X_4\}$.
\end{thmc3}

\begin{figure}[h]
	\begin{center}
		\includegraphics[scale=0.9]{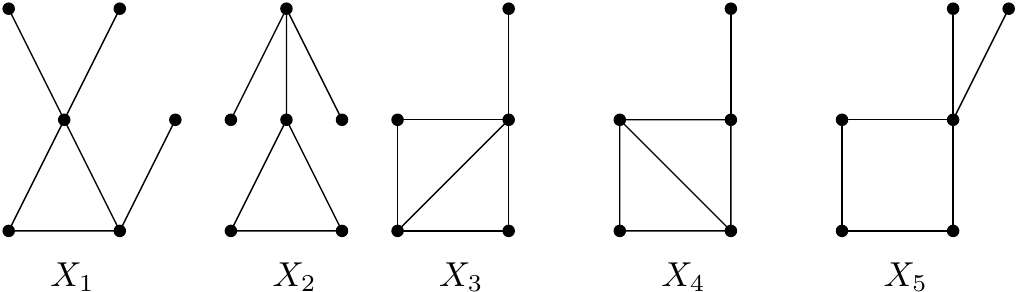}
	\end{center}
  \caption{The graphs $X_1,X_2,X_3,X_4,X_5$.}
  \label{fig-c3}
\end{figure}

This paper is organized as follows. In Section~\ref{section-minor}, we prove Theorem~\ref{thm:K4minor} and in Section~\ref{section-subgraph}, we prove Theorem~\ref{thm-c3}.
Section~\ref{section-subgraph} is further divided into three subsections.
Subsection~\ref{builder} and Subsection~\ref{painter} deals with the classes of graphs where Builder and Painter wins, respectively. 
Subsection~\ref{finaltouch} concludes Section~\ref{section-subgraph}.

For an edge $e$, we say that ``Painter cannot color $e$'' if there is a monochromatic copy of $H$ whether Painter colors $e$ red or blue; in other words, Builder wins the game no matter what color Painter uses on $e$. 
%since this is the end of the game, and Builder wins the game. 
In particular, we say that ``Painter cannot color $e$ red (blue)'' or ``Painter must color $e$ blue (red)'', if we already observed that Painter will eventually lose (a monochromatic copy of $H$ will appear) when Painter colors $e$ red (blue). 

\section{The online Ramsey game for $C_3$ on $K_4$-minor-free graphs}\label{section-minor}

%A graph is \emph{$2$-degenerate} if every subgraph has a vertex of degree at most $2$.
%A \emph{planar graph} is a graph that can be embedded in the plane without crossing edges.
%A graph is \emph{outerplanar} if it is planar and all vertices are incident with one face.
%We say a graph $H$ is a \emph{minor} (or \emph{$H$-minor}) of a graph $G$ if $H$ can be obtained from $G$ by deleting vertices, deleting edges, and contracting edges.

Grytczuk, Ha{\l}uszczak, and Kierstead~\cite{2004GrHaKi} proved that Builder wins the online Ramsey game for $C_3$ on $2$-degenerate planar graphs, 
%can force a monochromatic copy of $C_3$ on planar $2$-degenerate graphs
but Painter wins the online Ramsey game for $C_3$ on outerplanar graphs.
%Builder cannot force a monochromatic copy of $C_3$ on outerplanar graphs. 
%, which is the class of graphs that do not contain $K_{2, 3}$ and $K_4$ as a minor. 
%We want to investigate a superset $\mathcal{S}$ of outerplanar graphs such that Painter wins the online Ramsey game $(C_3, \mathcal{S})$. 
%Builer cannot force a monochromatic copy of $C_3$ on $\mathcal{S}$. 
We extend the class the graphs on which Painter is known to win the online Ramsey game for $C_3$.
Since a graph is outerplanar if and only if it does not contain $K_{2,3}$ and $K_4$ as a minor, we focus on $K_{2,3}$-minor-free graphs and $K_4$-minor-free graphs.
We show that Painter wins the online Ramsey game for $C_3$ on $K_{4}$-minor-free graphs, but Builder still wins the online Ramsey game for $C_3$ on $K_{2, 3}$-minor-free graphs. 
%We first observe that allowing a $K_4$-minor does not help the Painter, and then we prove that allowing a $K_{2, 3}$-minor (but still forbidding a $K_4$-minor) makes Builder be able to force a monochromatic copy of $C_3$ on the class. 

The following proposition shows that Builder wins the online Ramsey game for $C_3$ on $K_{2,3}$-minor-free graphs. 
Builder will use Strategy~\ref{str:builder}.
%can force a monochromatic copy of $C_3$ on $K_{2,3}$-minor free graphs.
%We say for a graph $G$, $G$ is \emph{red} and \emph{blue} if all edges of $G$ is colored by red and blue, respectively.

%They proved that all $k$-colorable graphs are unavoidable on $k$-colorable graphs;
%namely, they proved that $C_3$ is unavoidable on $3$-colorable graphs. 
%Yet, $C_3$ is avoidable on $K_4$-minor free graphs, which is a subclass of $3$-colorable graphs. 
%
%$$\{\mbox{outerplanar}\} \subsetneq \{K_4\mbox{-minor free}\} \subsetneq \{\mbox{$3$-colorable graphs}\}$$

%\marginpar{IC: treewidth 2 and $K_4$-minor free graphs are different.}

\begin{strategy}\label{str:builder}
Builder draws a copy of $K_{1, 5}$. Let $u$ be the vertex of degree $5$. 
By the pigeonhole principle, Painter will color at least three edges with the same color, say $uv_1,uv_2,uv_3$.
Builder draws the edges $v_1v_2$, $v_2v_3$, and $v_3v_1$.
\end{strategy}

\begin{proposition}\label{prop:k23}
Builder wins the online Ramsey game for $C_3$ on $K_{2,3}$-minor-free graphs.
%Builder can force a monochromatic copy of $C_3$ on $K_{2,3}$-minor free graphs.
\end{proposition}

%The following strategy is a winning strategy of Builder for the online Ramsey game $(C_3, K_{2,3}\text{-minor-free graphs})$.

\begin{proof}
Builder uses Strategy~\ref{str:builder}.
Assume $uv_1, uv_2, uv_3$ are colored red.
If Painter colors one of $v_1v_2, v_2v_3, v_3v_1$ red, then this creates a red $C_3$. 
% (say, $v_1v_2$ is red), then this creates a red $C_3$ with vertices $u$, $v_1$, and $v_{2}$.
%Therefore, Painter must color $v_1v_2$ blue.
%Similarly, Painter should color $v_2v_3$ and $v_1v_3$ blue.
Therefore Painter must color all of $v_1v_2, v_2v_3, v_3v_1$ blue, but then this creates a blue $C_3$ with vertices $v_1$, $v_2$, and $v_{3}$.

The graph resulting from Strategy~\ref{str:builder} has no $K_{2,3}$ as a minor.
Thus Builder wins the online Ramsey game for $C_3$ on $K_{2,3}$-minor-free graphs.
\end{proof}

Now, we will prove that Painter wins the online Ramsey game for $C_3$ on $K_{4}$-minor-free graphs. 
The key idea of this proof stemmed from the proof of Theorem~\ref{thm:outerplanar} in~\cite{2004GrHaKi}. 

%Builder cannot force a monochromatic copy of $C_3$ on $K_4$-minor free graphs. In order to prove this, we need some definitions and the following lemma.
%\[\{\mbox{outerplanar graphs}\} \subsetneq \{K_4\mbox{-minor-free graphs}\} \subsetneq \{\mbox{planar $2$-degenerate graphs}\}\]
%We say a colored cycle(path) is zero, red, blue if 
%the number of red edges minus the number of blue edges is 
%$0$, $1$, $-1$ (mod $3$), respectively.
%For a graph $G$ and its vertices $x,y$,
%we denote $f_P(x,y)$ the number of red edges minus the number of blue edges in the path from $x$ to $y$ on P.
%We denote $f_C$ the number of red edges minus the number of blue edges in the cycle C.
%Here is a definition of a branch set and a branch vertex. 

%Usual definition of the minor relation uses a vertex deletion, an edge deletion, and an edge contraction. 
%Here, we introduce another well known definition of the minor relation. 
%Given graphs $G$, $H$ and a vertex $u$ of $H$ where $H$ is a minor of $G$, 
Recall that a graph $G$ contains $H$ as a minor if
%Given a graph $G$ and its minor $H$,
there exists a set $\mathcal S$ of pairwise disjoint subsets of $V(G)$ satisfying the following:
\begin{itemize}
\item For every vertex $u$ of $H$, there is an element $S_u\in \mathcal S$ such that $G[S_u]$ is connected. 
\item For every edge $uv$ of $H$, there is an edge between $S_u$ and $S_v$. 
\end{itemize}
We call $S_u$ the \emph{branch set} of $u$ in an $H$-minor of $G$ for every vertex $u$ of $H$. 
%For a vertex $u$ of $H$, the \emph{branch set} of $u$ in an $H$-minor of $G$ is a set of (connected) vertices of $G$ that represents $u$ in the $H$-minor. 
When the branch set has one vertex, we also call it a \emph{branch vertex}.
For two vertices $x, y$ in $G$, an \emph {$x, y$-path} is a path in $G$ from $x$ to $y$.

\begin{lemma}\label{lem:ordering}
%Let $G$ be a $K_4$-minor free graph with an edge $xy$.
Let $xy$ be an edge of a $K_4$-minor-free graph $G$, and
%Let $P,Q$ be two paths from $x$ to $y$ in $G-e$ and 
let $P$ and $Q$ be two $x, y$-paths in $G-xy$.
For an integer $k\ge 3$, if $x=v_1, \ldots, v_k=y$ are the common vertices of $P$ and $Q$, then these vertices are in the same order on both $P$ and $Q$. 
%$v_1,v_2,\ldots,v_n$ be the common vertices of $X_4$ and $Q$ with $n>2$.
%If $\dist_P(x,v_1)<\dist_P(x,v_2)<\cdots<\dist_P(x,v_n)$,
%then $\dist_Q(x,v_1)<\dist_Q(x,v_2)<\cdots<\dist_Q(x,v_n)$.
%In other words, $v_1, \ldots, v_k$ are on the same order on $X_4$ and $Q$.
\end{lemma}

\begin{proof}
%Denote $P=w_1w_2\ldots w_m$.
The claim is trivial when $k=3$, so we may assume $k>3$.
%since if $k=3$, then there is nothing to prove.
By reordering the indices, let $v_1, \ldots, v_k$ be the order of these vertices on $P$.

We claim that for $j>i+1$, %and $i>1$ or $j<m$,
if there is a $v_i, v_j$-path $R$ in $G$ that is internally disjoint with $P$,
then there is no path from $\{v_{i+1},v_{i+2},\ldots,v_{j-1}\}$ to $V(P)\setminus\{v_i,v_{i+1},\ldots,v_j\}$ that is internally disjoint with $P$.
Suppose not. 
%Take a shortest $a, b$-path $P'$ where $a\in\{v_{i+1},v_{i+2},\ldots,v_{j-1}\}$ and $b\in V(P)\setminus\{v_i,v_{i+1},\ldots,v_j\}$.
Take an $a, b$-path $P'$ where $a\in\{v_{i+1},v_{i+2},\ldots,v_{j-1}\}$ and $b\in V(P)\setminus\{v_i,v_{i+1},\ldots,v_j\}$.
If $P'$ and $R$ share a vertex $z$, then $G$ has a $K_4$-minor where the branch vertices are $z, v_i, v_j, a$.
If $P'$ and $R$ are vertex disjoint, then $G$ has a $K_4$-minor where the branch vertices are $a, b, v_i, v_j$. 
%Then there exist six internally vertex-disjoint paths whose ends are in $a,b,v_i,v_j$ in $G$ 
%which gives $K_4$ as a minor.(WANT TO FIX THIS SENTENCE)

Thus, if $R$ is a subpath of $Q$, then $Q$ can never visit $v_{i+1},v_{i+2},\ldots,v_{j-1}$
because otherwise $Q$ will contain a subpath from $\{v_{i+1},v_{i+2},\ldots,v_{j-1}\}$  to $x$ or $y$.
This is a problem since $Q$ is an $x, y$-path and must go through all of $v_1, \ldots, v_k$.
%Therefore, we can easily conclude that if $\dist_P(x,v_1)<\dist_P(x,v_2)<\cdots<\dist_P(x,v_n)$,
%then $\dist_Q(x,v_1)<\dist_Q(x,v_2)<\cdots<\dist_Q(x,v_n)$.
Therefore, we conclude that $v_1, \ldots, v_k$ are in the same order on both $P$ and $Q$.
\end{proof}

Given two vertices $u, v$ on a path $P$, let $P[u, v]$ denote the subpath of $P$ from $u$ to $v$.
%A graph is \emph{$2$-edge-colored} if every edge is colored by red or blue.
For a $2$-edge-colored graph $H$, 
let $f(H)$ denote the number of red edges minus the number of blue edges in $H$ modulo $3$. 
%A $2$-edge-colored graph $H$ is (of type) \emph {zero, red}, and \emph {blue} if $f(H)$ is $0$, $1$, and $2$, respectively.
%A $2$-edge-colored graph $H$ is \emph{balanced}, \emph{RED}, and \emph{BLUE} if $f(H)$ is $0$, $1$, and $2$, respectively.
A $2$-edge-colored graph $H$ is \emph{zero}, \emph{positive}, and \emph{negative} if $f(H)$ is $0$, $1$, and $2$, respectively.
Given a $2$-edge-colored graph $G$, a zero cycle $C$ is \emph{good} if there exist two vertices $\alpha,\beta$ on $V(C)$ 
such that an $\alpha, \beta$-path on $C$ is zero and 
%there is an $\alpha,\beta$-path where its internal vertices are disjoint from $V(C)$.
%$G$ has a path from $\alpha$ to $\beta$ whose internal vertices are disjoint from $C$.
there exists an $\alpha, \beta$-path in $G$ whose internal vertices are disjoint from $V(C)$.

\begin{theorem}\label{thm:K4minor}
Painter wins the online Ramsey game for $C_3$ on $K_{4}$-minor-free graphs.
%Builder cannot force a monochromatic copy of $C_3$ on $K_4$-minor free graphs.
\end{theorem}

\begin{proof}
Assume Builder drew the edge $e=xy$ to the previous graph to obtain the current graph $G$, which is $2$-edge-colored except for $e$.
Since the initial graph has no edges, it suffices to show that if $G-e$ has a $2$-edge-coloring such that every zero cycle is good, then this coloring can be extended to $G$ so that every zero cycle is good.
Note that if every zero cycle is good, then there cannot be a monochromatic $C_3$, since a monochromatic $C_3$ is a zero cycle and cannot have a zero path as a
 subgraph.

Suppose whenever Painter tries to color $e$ red and blue in $G$, there arises a zero cycle $C^r$ and $C^b$, respectively, that is not good.
Let $P^r=C^r-e$ and $P^b=C^b-e$. 
Since $C^r$ and $C^b$ are zero cycles, $P^r$ is negative and $P^b$ is positive.
Let $x=v_1, v_2, \ldots, v_t=y$ be the common vertices of $P^r$ and $P^b$.
By Lemma~\ref{lem:ordering}, they are in the same order on $P^r$ and $P^b$.
Without loss of generality, let $v_1, \ldots, v_t$ be the ordering of these vertices on $P^r$ and $P^b$.
Note that $P^r[v_j,v_{j+1}]=P^b[v_j,v_{j+1}]$ might happen for some $j\in\{1, \ldots, t-1\}$, but there must exist an $i$ where $P^r[v_i,v_{i+1}]\neq P^b[v_i,v_{i+1}]$, 
since $P^r$ is negative and $P^b$ is positive. 
Fix such an $i$, and note that $P^r[v_i, v_{i+1}]$ and $P^b[v_i, v_{i+1}]$ are internally disjoint. 

We claim that %both $P^r[v_i,v_{i+1}]$ and $P^b[v_i,v_{i+1}]$ are not of type zero.
 both $P^r[v_i,v_{i+1}]$ and $P^b[v_i,v_{i+1}]$ are not zero.
Without loss of generality, assume $P^r[v_i,v_{i+1}]$ was zero.
Since $P^b[v_i,v_{i+1}]$ is a path from $v_i$ to $v_{i+1}$ whose internal vertices are disjoint from $V(C^r)$, this implies that
$C^r$ is a good cycle, which is a contradiction.

Now we claim that $P^r[v_i,v_{i+1}]$ and $P^b[v_i,v_{i+1}]$ are either both positive or both negative. 
Without loss of generality assume $P^r[v_i,v_{i+1}]$ is positive and $P^b[v_i,v_{i+1}]$ is negative. 
Since the cycle $D$ formed by $P^r[v_i,v_{i+1}]$ and $P^b[v_i,v_{i+1}]$ is zero even before Builder drew $e$, 
we know that $D$ is a good cycle by the induction hypothesis. 
Therefore, there are two vertices $\alpha,\beta$ on $D$ where 
an $\alpha, \beta$-path on $D$ is zero and
$G-e$ (also, $G$) has an $\alpha, \beta$-path whose internal vertices are disjoint from $V(D)$.
Note that this latter $\alpha, \beta$-path cannot share its internal vertices with $P^r$ and $P^b$ since this would create a $K_4$-minor.
If both $\alpha, \beta$ are on the same $P^j$ for some $j\in\{r,b\}$, then because there are two zero $\alpha, \beta$-paths (on $C^j$) and another internally disjoint $\alpha, \beta$-path, we can conclude $C^j$ is good, which is a contradiction. 
If $\alpha, \beta$ are on different paths of $P^r, P^b$, 
then $G$ contains $K_4$ as a minor where the branch vertices are $v_i, v_{i+1}, \alpha, \beta$, which is again a contradiction. 

Now we know that $P^r[v_i,v_{i+1}]$ and $P^b[v_i,v_{i+1}]$ are both positive or both negative for every $i\in\{1, \ldots ,t-1\}$, 
which implies that $P^r$ and $P^b$ are both positive or both negative, which contradicts that $P^r$ is negative and $P^b$ is positive.

Thus, Painter can color $e$ so that every zero cycle in $G$ is good, and hence there is no monochromatic $C_3$ in the coloring Painter produces. 
\end{proof}

We remark that the proof of Theorem~\ref{thm:K4minor} works for not only $K_4$-minor-free graphs, but also $K_4$-topological-minor-free graphs.

\section{The online Ramsey game for $C_3$ on $F$-free graphs}\label{section-subgraph}

In this section, we attempt to characterize all graphs $F$ such that Painter wins the online Ramsey game for $C_3$ on $F$-free graphs. %Builder cannot force a monochromatic copy of $C_3$ on $H$-free graphs.
We determine the winner of the game in all cases except when $F$ is the graph $X_5$, which is in Figure~\ref{fig-c3}.
Recall that we put the constraint that $F$ has no isolated vertices because the game is defined to have infinitely many isolated vertices as the initial state. 
Here is our main result. 

\begin{theorem}\label{thm-c3}
Let $X_1, \ldots, X_5$ be the graphs in Figure~\ref{fig-c3}.
Suppose that $F$ is a graph with no isolated vertices that is not isomorphic to $X_5$. 
Painter wins the online Ramsey game for $C_3$ on $F$-free graphs if and only if $F$ is isomorphic to a subgraph of a graph in $\{X_1, X_2, X_3, X_4\}$.
\end{theorem}
%
%\begin{figure}[h]
%	\begin{center}
%		\includegraphics[scale=0.7]{thm-c3.pdf}
%	\end{center}
%  \caption{The graphs $X_1,X_2,X_3,X_4,X_5$.}
%  \label{fig-c3_2}
%\end{figure}

%The remaining of the paper is devoted to proving Theorem~\ref{thm-c3}. 
%We divide this section into three subsections. 
%Subsection~\ref{builder} and Subsection~\ref{painter} investigates the classes of graphs Builder and Painter wins, respectively. 
%Subsection~\ref{finaltouch} finishes the section.

\subsection{When does Builder win the online Ramsey game for $C_3$ on $F$-free graphs?}\label{builder}

In this subsection, we provide three different classes where Builder wins the online Ramsey game for $C_3$. 
%This means that Builder wins the online Ramsey game for $C_3$ on $\mathcal{C}$.
We start by proving Lemma~\ref{lemma-c3}, which shows that we only need to consider $F$ to be a subgraph of the graph $X$, which is in Figure~\ref{fig-lemma-c3}.
%By Lemma~\ref{lemma-c3}, it remains to check the cases for $F$-free graphs when $F$ is isomorphic to a subgraph of $X$.
Then we investigate the classes of (1) $K_4$-free graphs, (2) $K_{1,5}$-free graphs, and (3) $Y$-free graphs where $Y$ is the graph in Figure~\ref{fig-c3-2}.

%strategies that forces Painter to create a monochromatic $C_3$.

\begin{figure}[h]
	\begin{center}
		\includegraphics[scale=1.1]{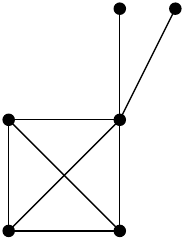}
	\end{center}
  \caption{The graph $X$.}
  \label{fig-lemma-c3}
\end{figure}

\begin{lemma}\label{lemma-c3}
Let $X$ be the graph in Figure~\ref{fig-lemma-c3}. 
%See Figure~\ref{fig-lemma-c3} for a definition of graph $X$.
If a graph $F$ is not isomorphic to a subgraph of $X$, then Builder wins the online Ramsey game for $C_3$ on $F$-free graphs. %, where $X$ is the graph in Figure~\ref{fig-lemma-c3}.
\end{lemma}
\begin{proof}
Builder uses Strategy~\ref{str:builder}.
Assume $uv_1, uv_2, uv_3$ are colored red.
If Painter colors one of $v_1v_2, v_2v_3, v_3v_1$ red, then this creates a red $C_3$. 
% (say, $v_1v_2$ is red), then this creates a red $C_3$ with vertices $u$, $v_1$, and $v_{2}$.
%Therefore, Painter must color $v_1v_2$ blue.
%Similarly, Painter should color $v_2v_3$ and $v_1v_3$ blue.
Therefore Painter must color all of $v_1v_2, v_2v_3, v_3v_1$ blue, but then this creates a blue $C_3$ with vertices $v_1$, $v_2$, and $v_{3}$.

There is no $F$ as a subgraph at every step of the game since the resulting graph is $X$ and $F$ is not isomorphic to any of the subgraphs of $X$. %See Figure~\ref{fig-lemma-c3}.
Hence, Builder wins the online Ramsey game for $C_3$ on $F$-free graphs.
\end{proof}
%\begin{proof}
%We will present a winning strategy for Builder. 
%
%Builder draws $K_{1,5}$ with five pairwise disjoint edges $vv_1,vv_2,vv_3,vv_4$ and $vv_5$. Painter will color at least three of them with the same color, we may assume that $vv_1,vv_2,vv_3$ are colored red. 
%
%Then Builder draws $v_1v_2,v_2v_3$ which must be colored blue. If Builder draws $v_3v_1$, a monochromatic $C_3$ is created when the edge is colored either red or blue. 
%
%There is no $H$ as a subgraph in every step of the game since the resulting graph is $X$ and $H$ is not isomorphic to any of the subgraphs of $X$. See Figure~\ref{fig-lemma-c3}.
%
%Hence, Builder wins the online Ramsey game $(C_3,H$-free graphs$)$.
%\end{proof} 

The following Proposition~\ref{prop-K_4-free} is a special case of a result in~\cite{2004GrHaKi}, and a more general theorem is proved in \cite{2009KiKo}. 
For the sake of completeness, we include a proof of Proposition~\ref{prop-K_4-free}.

\begin{figure}
	\begin{center}
		\includegraphics[scale=0.7]{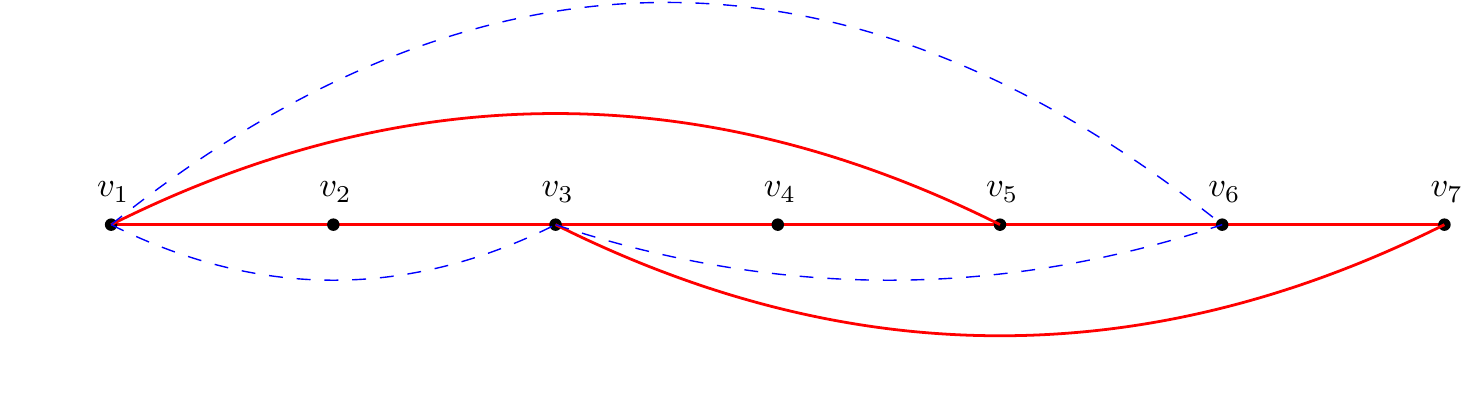}
	\end{center}
  \caption{A strategy for Builder to win the online Ramsey game for $C_3$ on $K_4$-free graphs.}
  \label{fig-c3k4}
\end{figure}

\begin{proposition}\label{prop-K_4-free}
Builder wins the online Ramsey game for $C_3$ on $K_4$-free graphs.
%Builder can force a monochromatic copy of $C_3$ on $K_4$-free graphs.
\end{proposition}
\begin{proof}
We will present a winning strategy for Builder. 

%Since Builder wins the online Ramsey game $(S,\mathcal D)$ where $S$ is a forest and $\mathcal D$ is the class of all forests by~\cite{2004GrHaKi}%can force a monochromatic copy of any forest on the class of forests by~\cite{2004GrHaKi}
%, we may assume that Builder can force Painter to create a monochromatic path of length six with seven vertices $v_1,v_2,\ldots,v_7$. 
Given a forest $S$, it is known that Builder wins the online Ramsey game for $S$ on the class of all forests by~\cite{2004GrHaKi}.
%Since Builder wins the online Ramsey game for $S$ on $\mathcal{D}$ where $S$ is a forest and $\mathcal D$ is the class of all forests by~\cite{2004GrHaKi}, 
Thus, we may assume that Builder has forced Painter to create a monochromatic path of length six while drawing a forest. 
We label the seven vertices on the path by $v_1,v_2,\ldots,v_7$ and suppose that these vertices on the path are in this order. 
Without loss of generality, assume the edges of the path are colored red. 
Note that there might be more edges incident with $v_i$ for $i\in\{1,\ldots ,7\}$, but since the whole graph is a forest, it is $K_4$-free.

Next, Builder draws $v_1v_5$ and $v_3v_7$. We claim that Painter must color both $v_1v_5$ and $v_3v_7$ red. Without loss of generality assume that $v_1v_5$ is colored blue. 
Now Builder draws both $v_1v_3$ and $v_3v_5$. Painter must color $v_1v_3$ blue, otherwise there is a red $C_3$ with three vertices $v_1,v_2,v_3$. Now Painter cannot color $v_3v_5$.
%Then Builder draws $v_1v_3$, which must be colored blue. Now if Builder draws $v_3v_5$, then a monochromatic $C_3$ is created when the edge is colored either red or blue. 
Therefore, both $v_1v_5$ and $v_3v_7$ must be colored red. 

%Next, Builder draws $v_3v_7$. If $v_3v_7$ is colored blue, then Builder draws $v_3v_5$, which must be colored blue. Then if Builder draws $v_5v_7$, then a monochromatic $C_3$ is created when the edge is colored either red or blue. Therefore, $v_3v_7$ must be colored red.
Finally, Builder draws three edges $v_1v_3$, $v_3v_6$, and $v_6v_1$. If Painter colors any of them red, then a red $C_3$ is created. 
Otherwise, Painter colors all of them blue, and this creates a blue $C_3$ with three vertices $v_1,v_3$, and $v_6$. See Figure~\ref{fig-c3k4}.

Four vertices of degree at least 3 appear only in the previous paragraph. 
It is easy to check that $K_4$ does not appear as a subgraph in this case, so $K_4$ does not appear as a subgraph at every step of the game.
Hence, Builder wins the online Ramsey game for $C_3$ on $K_4$-free graphs.
%Hence, Builder can force a monochromatic copy of $C_3$ on $K_4$-free graphs. 
\end{proof}

\begin{figure}
	\begin{center}
		\includegraphics[scale=0.7]{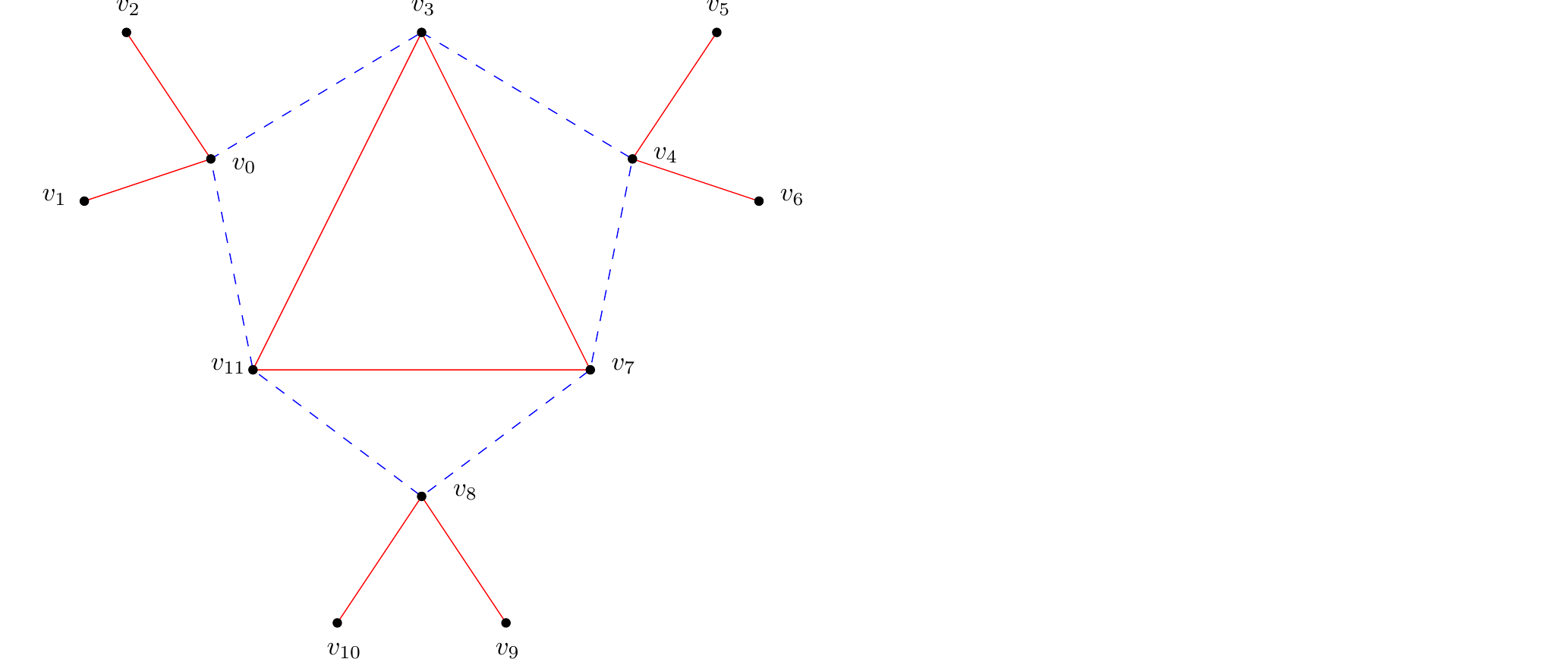}
	\end{center}
  \caption{A strategy for Builder to win the online Ramsey game for $C_3$ on $K_{1,5}$-free graphs.}%A graph for Proposition~\ref{prop-K_15-free}.}
  \label{fig-c3k15}
\end{figure}

The following proposition is implied by a result in~\cite{2011BuGrKiMiStWe} (see Proposition 4.2).
For completeness, we provide a proof here as well.

\begin{proposition}\label{prop-K_15-free}
Builder wins the online Ramsey game for $C_3$ on $K_{1,5}$-free graphs.
%Builder can force a monochromatic copy of $C_3$ on $K_{1, 5}$-free graphs.
\end{proposition}
\begin{proof}
We will present a winning strategy for Builder. 

Builder draws five pairwise disjoint induced copies of $K_{1,3}$. We claim that Painter must not create a monochromatic copy of $K_{1,3}$. 
%Note that each of them must be colored by two red edges and one blue edge or two blue edges and one red edge. 
Otherwise, without loss of generality, assume that there is a red $K_{1,3}$. Now, Builder draws $K_4$ containing the red $K_{1,3}$ as a subgraph. If Painter colors any of the newly drawn edges red, then a red $C_3$ is created. Otherwise, Painter colors all of the newly drawn edges blue, and a blue $C_3$ is created. 

Therefore, since there is no monochromatic copy of $K_{1,3}$, we may assume that at least three of the five pairwise disjoint induced copies of $K_{1,3}$ contain two red edges and one blue edge; 
let these copies of $K_{1,3}$ be $S_0,S_1,S_2$ where $V(S_i)=\{v_{4i},v_{4i+1},v_{4i+2},v_{4i+3}\}$ and $E(S_i)=\{v_{4i}v_{4i+1},v_{4i}v_{4i+2},v_{4i}v_{4i+3}\}$ for $i\in\{0,1,2\}$ while $v_0v_3, v_4v_7$, and $v_8v_{11}$ are blue, and all other edges in $E(S_0)\cup E(S_1)\cup E(S_2)$ are red. 

Next, Builder draws $v_3v_4, v_7v_8$, and $v_{11}v_0$. We claim that Painter must color all these edges blue. Suppose without loss of generality that Painter colors $v_3v_4$ red. Then Builder draws $v_3v_5, v_5v_6$, and $v_6v_3$. If Painter colors any of them red, then a red $C_3$ is created. If Painter colors all of them blue, then this creates a blue $C_3$ with vertices $v_3,v_5$, and $v_6$. 

Therefore we may assume that Painter colors $v_3v_4, v_7v_8$, and $v_{11}v_0$ blue. Finally, Builder draws $v_3v_7, v_7v_{11}$, and $v_{11}v_3$. If Painter colors any of them blue, then a blue $C_3$ is created. If Painter colors all of them red, then this creates a red $C_3$ with vertices $v_3, v_7$, and $v_{11}$. See Figure~\ref{fig-c3k15}.

It is easy to check that $K_{1,5}$ does not appear as a subgraph at every step of the game.
Hence, Builder wins the online Ramsey game for $C_3$ on $K_{1,5}$-free graphs.
%Hence, Builder can force a monochromatic copy of $C_3$ on $K_{1, 3}$-free graphs. 
\end{proof}

\begin{figure}[h]
	\begin{center}
		\includegraphics[scale=0.7]{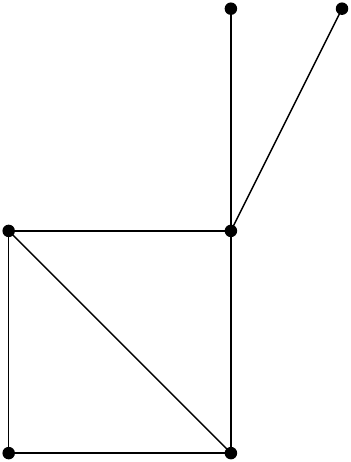}
	\end{center}
  \caption{The graph $Y$.}
  \label{fig-c3-2}
\end{figure}

\begin{lemma}\label{lemma-A-free}
Let $Y$ be the graph in Figure~\ref{fig-c3-2}. 
While playing the online Ramsey game for $C_3$ on $Y$-free graphs, 
Builder can force Painter to create either a monochromatic copy of $C_3$ or a blue edge $xy$ with $\deg(x)=1$ and $\deg(y)\leq2$.
\end{lemma}
\begin{proof}
This can be proven by letting Builder draw an edge and extend it to a path of length $4$. 
At any moment, if Painter colors any of these edges blue, then that creates the blue edge we seek. 
%with one end of degree $1$ and the other end of degree at most $2$. 
Otherwise, we may assume Painter produced a red path of length $4$. %if Painter colors those edges of the path in red, then Builder repeats this until Builder draws a path of length 4. 
Let $P$ be such a red path with vertices $x_1,x_2,x_3,x_4$, and $x_5$ in this order on $P$.

Now, Builder draws two edges $x_2x_6$ and $x_4x_6$ with a new vertex $x_6$. 
We claim that Painter must color both $x_2x_6$ and $x_4x_6$ with the color blue. 
Without loss of generality, suppose Painter colors $x_2x_6$ red. 
Now, Builder draws $x_1x_3,x_3x_6$, and $x_6x_1$. 
If Painter colors any of these edges red, then there is a red $C_3$. 
If Painter colors all of these edges blue, then this creates a blue $C_3$. Therefore, Painter must color $x_2x_6$ and $x_4x_6$ blue. 

Finally, Builder draws $x_2x_4$. Whenever Painter colors $x_2x_4$ red or blue, this creates a monochromatic copy of $C_3$.% See Figure~\ref{fig-c3-2}.

It is easy to check that $Y$ does not appear as a subgraph at every step of the game.
Hence, Builder can force Painter to create either a monochromatic copy of $C_3$ or a blue edge $xy$ with $\deg(x)=1$, $\deg(y)\leq2$, while playing the online Ramsey game for $C_3$ on $Y$-free graphs.
\end{proof}

\begin{proposition}\label{prop-A-free}
Let $Y$ be the graph in Figure~\ref{fig-c3-2}. 
Builder wins the online Ramsey game for $C_3$ on $Y$-free graphs.
%Builder can force a monochromatic copy of $C_3$ on $A$-free graphs while graph $A$ is as defined in Figure~\ref{fig-c3-2}.
\end{proposition}

\begin{proof}
We will present a winning strategy for Builder. 

%Without loss of generality, Painter colors four edges red when Builder draws seven pairwise disjoint edges. 
Builder draws seven pairwise disjoint edges. 
By the pigeonhole principle, Painter colors at least four edges with the same color.
Without loss of generality, assume $v_1w_1$, $v_2w_2$, $v_3w_3$, and $v_4w_4$ are red edges.
%Say those edges $e_1, e_2, e_3, e_4$ where $e_i=v_iw_i$. 

Next, Builder draws the four edges $vv_i$ for $i\in\{1,2,3,4\}$ with a new vertex $v$. 
We claim that Painter must color two of them red and the other two blue. 
%Suppose Painter colors three of them red and suppose $vv_1,vv_2$, and $vv_3$ are such edges. 
Suppose Painter colors $vv_1$, $vv_2$, and $vv_3$ red.
Now Builder draws $v_1v_2, v_2v_3$, and $v_3v_1$. 
If Painter colors any of them red, then a red $C_3$ is created. 
If Painter colors all of these edges blue, then this creates a blue $C_3$ with vertices $v_1,v_2$, and $v_3$. 
Therefore, we may assume that $vv_1,vv_2$ are red and $vv_3,vv_4$ are blue. 

Next, Builder draws $w_1w_2$. 
Suppose Painter colors $w_1w_2$ blue. 
Now, Builder draws $vw_1$ and $vw_2$. 
If Painter colors any of these edges red, then a red $C_3$ is created. 
If Painter colors both $vw_1$ and $vw_2$ blue, then a blue $C_3$ with vertices $v,w_1$, and $w_2$ is created. 
Therefore we may assume that Painter colors $w_1w_2$ red. 
%when the edge is colored either red or blue. Therefore we may assume that $w_1w_2$ is red. 

Now, Builder forces Painter to create a blue edge $xy$ with $\deg(x)=1$ and $\deg(y)\leq2$, which is possible by Lemma~\ref{lemma-A-free}.
%This can be made by drawing a long path from one vertex and drawing from the vertex while making it connected. If there is a red path of arbitrarily long length, then let $P$ be the red path with vertices $x_1,x_2,x_3,x_4,x_5$ with $deg(X_2)=1$ and $deg(x_i)=2$ for $i\in\{2,3,4,5\}$. Then Builder can win by drawing two edges $x_2x_6, x_4x_6$ with new vertex $x_6$ since they must be colored in blue and by drawing $x_2x_4$, since monochromatic $C_3$ is created when the edge is colored either red or blue. 
Next, Builder draws $xw_1$ and $xw_2$. 
We claim that Painter must color these edges blue. 
Without loss of generality, suppose $xw_1$ is colored red. 
Then Builder draws two more edges $xv_1$ and $v_1w_2$. 
If Painter colors any of $xw_2,xv_1$, and $v_1w_2$ red, then there is a red $C_3$. 
If Painter colors all of them blue, then this creates a blue $C_3$ with vertices $x,v_1$, and $w_2$. 
%Builder draws $xv_1$ and $v_1w_$, which also must be blue. If Builder draws $v_1w_2$, then monochromatic $C_3$ is created when the edge is colored either red or blue. 
Therefore, Painter must color $xw_1$ and $xw_2$ blue. 

Finally, Builder draws $yw_1$ and $yw_2$. If Painter colors any of them blue, then there is a blue $C_3$. If Painter colors all of them red, then this creates a red $C_3$ with vertices $y,w_1$ and $w_2$. % which must be colored red. If Builder draws $yw_2$, then monochromatic $C_3$ is created when the edge is colored either red or blue. 
See Figure~\ref{fig-c3A}.

It is easy to check that $Y$ never appears as a subgraph at every step of the game.
Hence, Builder wins the online Ramsey game for $C_3$ on $Y$-free graphs.
%Hence, Builder can force a monochromatic copy of $C_3$ on $A$-free graphs.
\end{proof}

\begin{figure}
	\begin{center}
		\includegraphics[scale=0.7]{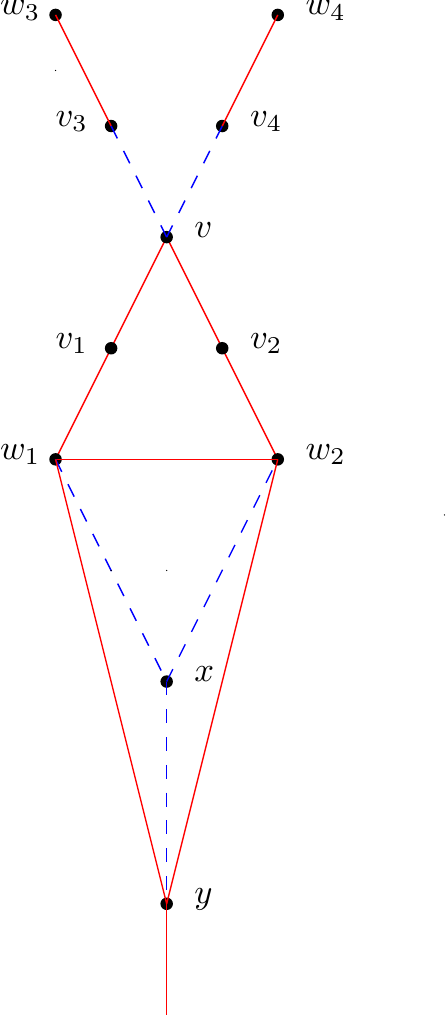}
	\end{center}
  \caption{A strategy for Builder to win the online Ramsey game for $C_3$ on $Y$-free graphs.}%A graph for Proposition~\ref{prop-A-free}. }
  \label{fig-c3A}
\end{figure}

\subsection{When does Painter win the online Ramsey game for $C_3$ on $F$-free graphs?}\label{painter}

%With this Lemma~\ref{lemma-c3}, it is enough to check for cases of $H$-free graphs when $H$ is a subgraph of $X$.

In this section, we will prove that Painter wins the online Ramsey game for $C_3$ on $F$-free graphs for various $F$. %Builder cannot force a monochromatic copy of $C_3$ on $H$-free graphs for some fixed $H$. 
Recall that by Lemma~\ref{lemma-c3}, we only need to consider $F$ to be a subgraph of the graph $X$, which is in Figure~\ref{fig-lemma-c3}.
For a fixed $F$, it is sufficient to provide a strategy for Painter so that a monochromatic $C_3$ does not appear forever on $F$-free graphs. 
We will provide three different winning strategies for Painter for three different $F$.
%Before starting the proof of Theorem~\ref{thm-c3}, we introduce the following lemma.

\begin{strategy}\label{strategy-N}
Painter colors each new edge red, unless doing so creates a red $K_{1, 3}$, a red $C_3$, or a red $C_4$, in which case the new edge is colored blue. 
\end{strategy}

\begin{proposition}\label{prop-N-free}
Let $X_1$ be the graph in Figure~\ref{fig-c3}.
Painter wins the online Ramsey game for $C_3$ on $X_1$-free graphs.
%Builder cannot force a monochromatic copy of $C_3$ on $X_1$-free graphs. 
\end{proposition}
\begin{proof}
Painter will use Strategy~\ref{strategy-N}.
We claim that Painter can always color the new edge $e=xy$ with Strategy~\ref{strategy-N}.
Let $G$ be the new graph when Builder draws $e$.  
We will use induction on the number of edges. 
The base case is trivial. 

By the induction hypothesis, we may assume that there is no red $K_{1, 3}$, no red $C_3$, no red $C_4$, and no blue $C_3$ in $G-e$.
The strategy fails when coloring $e$ blue results in a blue $C_3$ and coloring $e$ red results in a red $K_{1, 3}$, a red $C_4$, or a red $C_3$. 
Let $x, y, z$ be the vertices of the blue $C_3$ when $e$ is colored blue.
%so that $xz$ and $zy$ are blue. 
We will prove that if the strategy fails, then $G$ has $X_1$ as a subgraph, which is a contradiction, and thus the strategy does not fail. 
We will divide the cases according to which red subgraph appears when Painter colors $e$ red.%It is enough to show for the following three cases, when red $C_3$, red $K_{1,3}$ and red $C_4$ is created by coloring $e$ red. 

\paragraph{Case 1}
Assume a red $C_3$ is created when Painter colors $e$ red, and let $w$ be the third vertex of this red $C_3$. 
Since Painter colored neither $xz$ nor $zy$ red, coloring each of $xz$ and $yz$ red must have created a red $C_4$, a red $C_3$, or a red $K_{1, 3}$ in $G-e$.
We will show that a red $C_3$ or a red $C_4$ cannot be created by coloring either $xz$ or $yz$ red. 
Without loss of generality, let us consider $xz$. 
%Now, we proceed on two subcases that coloring $xz$ red creates one of a red $C_4, C_3$. Then finish this case by using two subcases. 

If coloring $xz$ red resulted in a red $C_4$ with vertices $x, s, t, z$ in cyclic order, then $t\not=y$ and $s\not=y$, since in $G-e$, the edge $yz$ is blue and $e$ does not exist.
We also know that $t\not=w$, since otherwise $G-e$ has a red $K_{1, 3}$ as a subgraph, which is a contradiction to the induction hypothesis. 
 If $s=w$, then it must be that $t=y$ in order for $G-e$ to not have a red $K_{1,3}$, but this contradicts that $t\not=y$.
This implies that $s, t\not\in\{x, y, z, w\}$, which means $G$ has $X_1$ as a subgraph, which is a contradiction.

If coloring $xz$ red resulted in a red $C_3$ %centered at $z$ 
with vertices $x, z, u$, then $u\neq w$, since otherwise $G- e$ has a red $K_{1,3}$ as a subgraph, which contradicts the induction hypothesis. 
This implies that $u\not\in\{x, y, z, w\}$. 
Now, $y$ and $z$ cannot have neighbors outside of $\{u, x, y, z, w\}$ since that would create a copy of $X_1$ in $G$.
There is no red edge between $u$ and $w$ because that would create a red $C_3$ in $G- e$. 
Since either a red $yu$ or a red $zw$ would create a red $K_{1, 3}$, neither $y$ nor $z$ can have more incident red edges, which means $yz$ could have been colored red, which is a contradiction. 

This boils down to the case where both $xz$ and $zy$ were colored blue because coloring either one red would create a red $K_{1, 3}$. 
Since $zw$ cannot be a red edge (creates a red $K_{1, 3}$ in $G- e$) and $z$ cannot have two neighbors outside of $\{x, y, w\}$ (creates a copy of $X_1$ in $G$), each of $x$ and $y$ have a neighbor $x'$ and $y'$, respectively, such that $xx'$ and $yy'$ are red.
It cannot be that $x'=y'$, since this creates a red $C_4$ with vertices $x, w, y, x'$ in $G- e$. 
If $x'\neq y'$, then this creates a copy of $X_1$ in $G$. 
In either case, we obtain a contradiction.

\paragraph{Case 2}
Assume a red $K_{1, 3}$ is created when Painter colors $e$ red, and without loss of generality let $x, y, u, v$ be the vertices of the red $K_{1, 3}$ so that $xy, ux, xv$ are red edges. 
Now, $z$ and $y$ cannot have neighbors outside of $\{x, y, z, u, v\}$ since that would create a copy of $X_1$.
%This implies that $zu$ and $zv$ cannot both be red edges, since that would create a red $C_4$ with vertices $z, u, x, v$. Similarly, $yu$ and $yv$ cannot both be red edges.
This implies that each of $z$ and $y$ cannot have two red edges incident to it, since that would create a red $C_4$, with vertices $z, u, x, v$.
Also, $uv$ cannot be a red edge since $G- e$ would have a red $C_3$, with vertices $u, v, x$.
Since $zy$ was not colored with red, coloring $zy$ with red must create a red $K_{1, 3}$, a red $C_3$, or a red $C_4$ in $G- e$. 
The only possible case is when coloring $zy$ with red creates a red $C_3$, which implies that either $u$ or $v$ is a vertex of this red $C_3$, which implies the existence of a red $K_{1, 3}$ in $G- e$, which is a contradiction.

\paragraph{Case 3}
Assume a red $C_4$ is created when Painter colors $e$ red, and let $xx', x'y', y'y$ be the red edges of this red $C_4$ other than $e$.
%Now $x$ and $y$ cannot have a neighbor $v$ where $xv$ and $yv$, respectively, are red since either it would create a copy of $X_1$ or a red $K_{1, 3}$ in $G$.
%Now $x$ and $y$ cannot have a neighbor $v$ where $xv$ and $yv$ is red, respectively, since either it would create a copy of $X_1$ or a red $C_3$ in $G$.
Now, neither $x$ nor $y$ can have a neighbor outside of $\{x, y, x', y', z\}$ since this would create a copy of $X_1$ in $G$.
Also, $x'$ and $y'$ cannot have a neighbor $v\not\in\{x, x', y', y\}$ where $x'v$ and $y'v$ is red, respectively, since this would create a red copy of $K_{1, 3}$ in $G- e$.
%Also, $z$ cannot have a neighbor $v\in\{x', y'\}$ where $zv$ is red since it would mean that there is a red $K_{1, 3}$ in $G- e$. 
Since Painter colored neither $xz$ nor $yz$ red, coloring each of $xz$ and $yz$ red must create a red $K_{1, 3}$, a red $C_4$, or a red $C_3$. 
The only possible case is when there is a red $K_{1, 3}$ centered at $z$ when Painter colors $xz$ or $yz$ red. 
In particular, $z$ must have two neighbors $z', z''$ outside of $\{x, x', y, y'\}$ where $zz'$ and $zz''$ are red edges. 
Yet, this creates a copy of $X_1$, which is a contradiction.

Therefore, Strategy~\ref{strategy-N} works and thus Painter wins the online Ramsey game for $C_3$ on $X_1$-free graphs.
\end{proof}

Before starting the proof for the case of $X_2$-free graphs, we define some \emph {``good"} subgraphs of a graph. 
We say a subgraph $H$ of $G$ that is isomorphic to either $K_{1, 3}$ or $C_4$ is \emph {good} if $H$ is red, and there exists a subgraph $I$ of $G$ where $H$ is a subgraph of $I$ in such a way that $I$ is isomorphic to one of the graphs in Figures~\ref{fig-c3M_1} and \ref{fig-c3M_2}, where the thick edges correspond to the edges of $H$; 
moreover, for $i\in\{1, \ldots, 5\}$, we say \emph{$H$ is good by property $A_i$ (or $B_i$)} to mean that the corresponding $I$ is isomorphic to the graph labeled $A_i$ (or $B_i$) in Figures~\ref{fig-c3M_1} and \ref{fig-c3M_2}.
We also say $H$ is {\it good} if $H$ is good because of multiple properties. 
For example, if $H$ satisfies the property $A_1$, then $H$ is isomorphic to $K_{1,3}$ and the vertex of degree $3$ of $G[V(H)]$ has degree at least $5$ in $G$. 
We say that a red subgraph $H$ of $G$ that is isomorphic to either $K_{1,3}$ or $C_4$ is \emph {bad} if it is not good. 
Note that if a subgraph $H$ is bad, then all of its edges are red. 

The idea is that we want to forbid $K_{1, 3}$ and $C_4$ in the graph as much as we can, but we allow copies of $K_{1, 3}$ and $C_4$ if we can guarantee that there is some structure we can utilize. 

\begin{figure}[h]
	\begin{center}
		\includegraphics[scale=0.7]{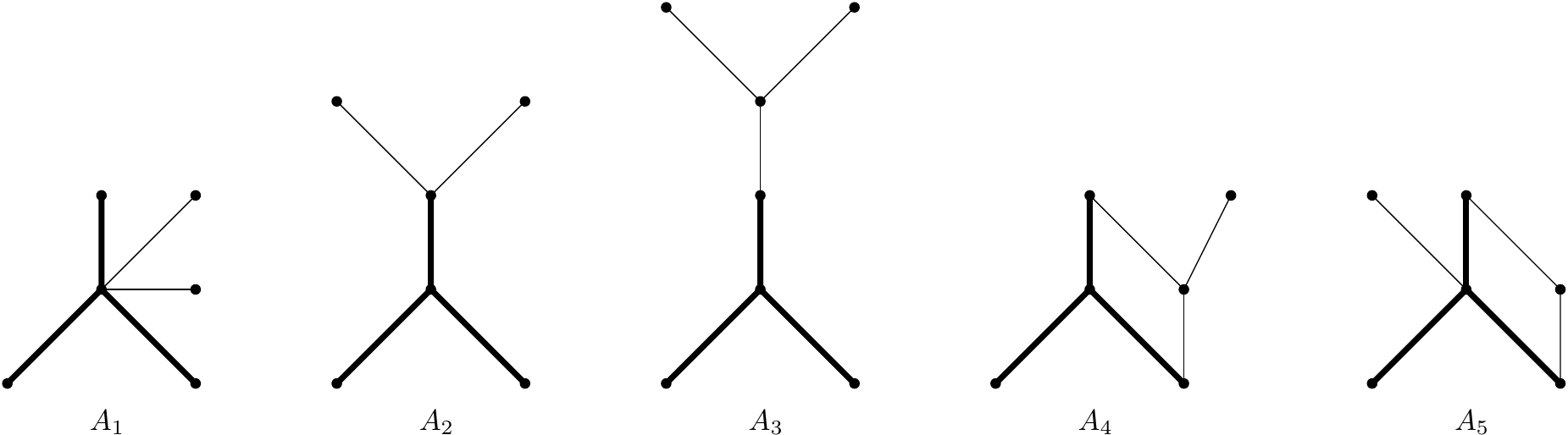}
	\end{center}
  \caption{The five good $K_{1,3}$'s.}
  \label{fig-c3M_1}
\end{figure}

\begin{figure}[h]
	\begin{center}
		\includegraphics[scale=0.6]{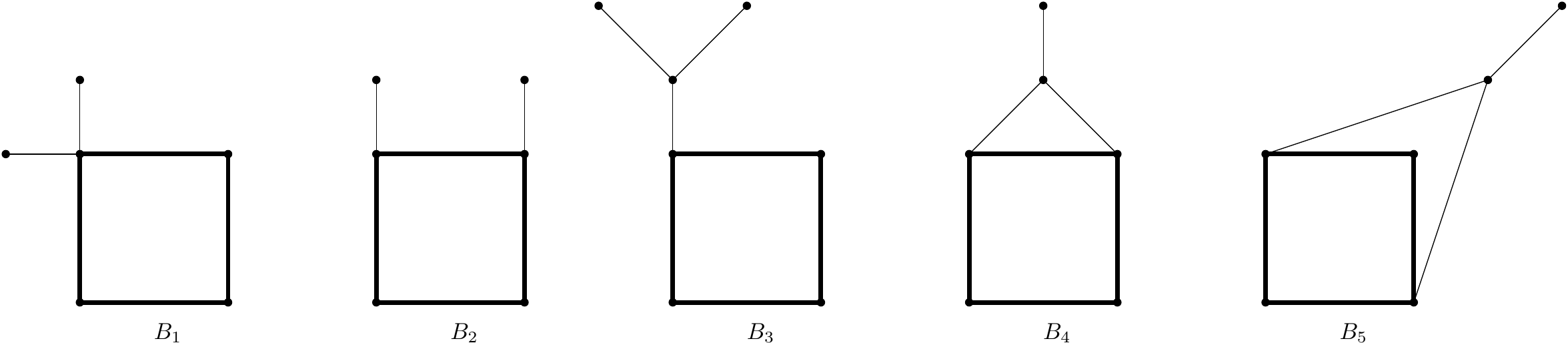}
	\end{center}
  \caption{The five good $C_4$'s.}
  \label{fig-c3M_2}
\end{figure}

%% updated this figure? think lines have a different meaning in figures 7 and 8

%\begin{lemma}\label{jisu-lemma}
%Let $X_2$ be the graph in Figure~\ref{fig-c3}.
%For every good $K_{1,3}$ with vertices $v, v_1, v_2$, and $v_3$ while $v$ is the center of $K_{1,3}$, Builder cannot draw all three edges $v_1v_2, v_2v_3$, and $v_3v_1$ in the online Ramsey game for $C_3$ on $X_2$-free graphs.
%\end{lemma}

\begin{lemma}\label{jisu-lemma}
Let $X_2$ be the graph in Figure~\ref{fig-c3}.
Let $G$ be a graph that has a good $K_{1,3}$ with vertices $v, v_1, v_2, v_3$ where $v$ is the vertex of degree $3$. 
If $v_1v_2$, $v_2v_3$, and $v_3v_1$ are edges in $G$, then $G$ contains $X_2$ as a subgraph.
\end{lemma}

\begin{proof}
See Figure~\ref{jisu-lemma-fig}. It is easy to check that $G$ has $X_2$ as a subgraph in each case.
\end{proof}
\begin{figure}[h]
	\begin{center}
		\includegraphics[scale=0.7]{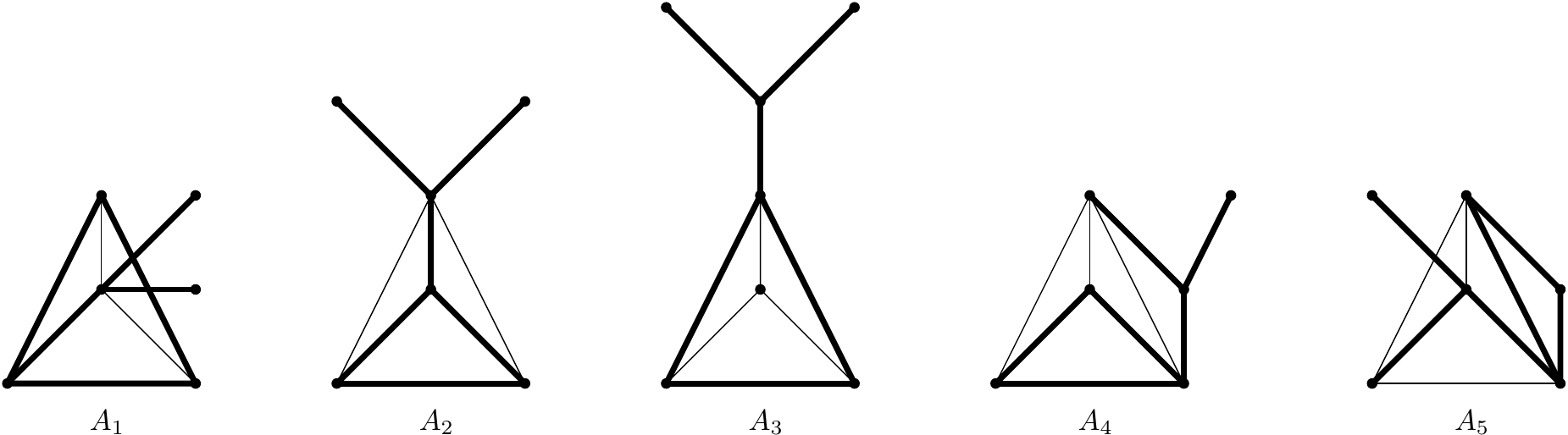}
	\end{center}
  \caption{Observation for the proof of Lemma~\ref{jisu-lemma}.}
  \label{jisu-lemma-fig}
\end{figure}

\begin{strategy}\label{strategy-M}
Painter colors each new edge red, unless doing so creates a red $C_3$, a bad $K_{1, 3}$, or a bad $C_4$, in which case the new edge is colored blue. 
\end{strategy}

\begin{proposition}\label{prop-M-free}
Let $X_2$ be the graph in Figure~\ref{fig-c3}.
Painter wins the online Ramsey game for $C_3$ on $X_2$-free graphs.
%Builder cannot force a monochromatic copy of $C_3$ on $X_2$-free graphs. 
\end{proposition}
\begin{proof}
Painter will use Strategy~\ref{strategy-M}.
We claim that Painter can always color the new edge $e=xy$ with Strategy~\ref{strategy-M}. 
Let $G$ be the new graph when Builder draws $e$. 
We will use induction on the number of edges. 
The base case is trivial.

By the induction hypothesis, we may assume that none of a red $C_3$, a bad $K_{1, 3}$, or a bad $C_4$ exists in $G- e$. 
The strategy fails when coloring $e$ blue results in a blue $C_3$ and coloring $e$ red results in a red $C_3$, a bad $K_{1, 3}$, or a bad $C_4$. 
Let $z$ be the vertex of the blue $C_3$ so that $xz$ and $zy$ are blue. 
Note that every blue edge has at least two red edges incident with it in $G$ while Painter uses Strategy~\ref{strategy-M}. 
We will prove that if the strategy fails, then $G$ has $X_2$ as a subgraph, which is a contradiction, and thus the strategy does not fail. 
We will divide the cases according to which red graph appears when Painter colors $e$ red.
%It is enough to show for the following three cases, when red $C_3$, bad $K_{1,3}$ and bad $C_4$ is created by coloring $e$ red. 

\paragraph{Case 1}
Assume a red $C_3$ is created when Painter colors $e$ red, and let $w$ be the third vertex of this red $C_3$. 
Since Painter did not color $xz$ and $yz$ red, coloring any of $xz$ and $yz$ red must have created a red $C_3$, a bad $C_4$, or a bad $K_{1, 3}$. 
By Lemma~\ref{jisu-lemma}, we may assume that there is no red edge between $z$ and $w$.  
Now, we consider three subcases where coloring $xz$ red creates one of a red $C_3$, a bad $K_{1, 3}$, or a bad $C_4$. 

%Since Painter did not color both $xz$ and $zy$ red, coloring each of $xz$ and $yz$ red must have created a red $K_{1, 3}, C_4$, or $C_3$ in $G-e$.

\paragraph{Subcase 1-1} 
Assume that coloring $xz$ red creates a red $C_3$ with vertices $x,z$, and $u$. 
Since we assumed that there is no red edge between $z$ and $w$, we know that $u\neq w$. 
By Lemma~\ref{jisu-lemma}, we may assume that there is no red edge between $y$ and $u$. 
Moreover, $y$ and $z$ cannot have neighbors outside of $\{x, y, z, u, w\}$, since $G$ cannot have $X_2$ as a subgraph. 
%otherwise $G$ has $X_2$ as a subgraph. 
However, this is a contradiction because Painter must have colored $yz$ red (instead of blue) since this does not create any of a bad $K_{1,3}$, a bad $C_4$, or a red $C_3$.
%, but $yz$ is colored blue. 
Note that although there can be an edge $uw$ in $G- e$, Painter could not color $uw$  red since this creates a red $C_3$ in $G- e$. 
%First, assume that coloring $xz$ red creates a red $C_3$. Let $u$ be the third vertex of this red $C_3$ which is not $x, z$. By previous assumption, $u\neq w$. Then by Lemma~\ref{jisu-lemma}, we may also assume that there is no red edge between $y$ and $u$. Note that $y,z$ cannot have neighbors outside of $\{x, y, z, u, w\}$, since otherwise, there is $X_2$ as a subgraph in $G$. However, this is a contradiction because coloring $yz$ red does not create any of bad $K_{1,3}, C_4$ or red $C_3$. Note that $uw$ cannot be red since this creates red $C_3$. 
 
\paragraph{Subcase 1-2}
Assume that coloring $xz$ red creates a bad $C_4$, say $R$, with vertices $x,u,v$, and $z$ in cyclic order. 
Since there is no red edge between $z$ and $w$, we know that $v\neq w$. 

Suppose $u\neq w$.
Note that $u, v$, and $z$ cannot have neighbors outside of $\{x,y,z,u,v,w\}$ and $E(G)$ has none of $vy, vx, vw$, and $uz$, otherwise $G$ has $X_2$ as a subgraph. Therefore, there was no red $C_3$ when Painter colored $yz$ red. 

If there was a bad $C_4$ when Painter colored $yz$ red, then the only possible case is when the bad $C_4$ consists of vertices $u, v, y$, and $z$ since $u, v$, and $z$ has no neighbor outside of $\{x, y, z, u, v, w\}$. Note that there is a red $K_{1,3}$ with vertices $u, v, x$, and $y$. If $\{x, y\}$ has no neighbors outside of $\{x, y, z, u, v\}$, then this red $K_{1,3}$ must be bad, which is a contradiction. Therefore, whenever $xz$ is drawn later than $yz$ or $yz$ is drawn later than $xz$, the later one must be colored red since the corresponding red $C_4$ is actually good. 

The only remaining reason that Painter colored $yz$ blue is that there are two red edges incident with $y$ so that coloring $yz$ red creates a bad $K_{1,3}$, say $S$. There are two cases: when there is a red edge between $y$ and $u$ so that $E(S)=\{yz,yu,yw\}$ and when there is no red edge between $y$ and $u$ but there is a red edge $ys$ with a new vertex $s$ so that  $E(S)=\{yz, ys, yw\}$. %It is easy to check that for both cases, one of $R$ and $S$ is good when Painter colors $xz$ and $zy$ red, which is a contradiction. Note that $xw$ may not be drawn at that step of the game.

\begin{itemize}
\item When there is a red edge between $y$ and $u$ so that $E(S)=\{yz,yu,yw\}$.

\begin{itemize} 
\item
Suppose Builder drew $xz$ later than $yz$. 
Then $R$ is good by property $B_5$, which is a contradiction. 
\item
Suppose Builder drew $yz$ later than $xz$.
Then $S$ is good by property $A_2$, which is a contradiction.  
\end{itemize}
\item When there is no red edge between $y$ and $u$ but there is a red edge $ys$ with a new vertex $s$ so that  $E(S)=\{yz, ys, yw\}$.
\begin{itemize}
\item Suppose Builder drew $xz$ later than $yz$. Then $R$ is good by property $B_3$, which is a contradiction. 
\item Suppose Builder drew $yz$ later than $xz$. Then $S$ is good by property $A_2$, which is a contradiction.  
\end{itemize}

Note that for both cases, $xw$ may not be drawn at each step of the game.
\end{itemize}

Now suppose $u=w$. It is easy to check that $v, w$, and $z$ cannot have neighbors outside of $\{v, w, x, y, z\}$, since otherwise $G$ has $X_2$ as a subgraph. Since a red $K_{1,3}$ with vertices $x, y, w, v$ must be good, $\{x, y\}$ must have at least one neighbor outside of $\{v, w, x, y, z\}$. Note that this is only true for those steps of the game in which the red $K_{1,3}$ has already been drawn. 

\begin{itemize}
\item
Suppose that $yz$ is drawn later than $xz$. 

\begin{itemize}
\item Coloring $yz$ red cannot create a red $C_3$ since $z$ cannot have neighbors outside of $\{v, w, x, y, z\}$.
\item Coloring $yz$ red cannot create a bad $C_4$ since the only possible red $C_4$ is of vertices $v, w, y$, and $z$. Since there is a red $K_{1,3}$, $\{x, y\}$ must have at least one neighbor outside of $\{v, w, x, y, z\}$ and this implies that the red $C_4$ is good. 
\item Coloring $yz$ red cannot create a bad $K_{1,3}$ since $z$ cannot have neighbors outside of $\{v, w, x, y, z\}$. Even if there are two red edges $ys$ and $yt$ for vertices $s$ and $t$ (one of $s$ and $t$ may be equal to $w$, but not to $v$ or $x$), the red $K_{1,3}$ with vertices $s, t, y$, and $z$ is good by property $A_2$. 

Note that there are no red edges between $z$ and $w$, between $v$ and $y$, between $v$ and $x$, or between $x$ and $y$.  
\end{itemize}

\item
Suppose that $xz$ is drawn later than $yz$. 
Now, there are two cases: when coloring $yz$ red created a bad $C_4$ or when coloring $yz$ red created a bad $K_{1,3}$. Note that coloring $yz$ red cannot create a red $C_3$. 

\begin{itemize}
\item If coloring $yz$ red would have created a bad $C_4$, then the vertices of this $C_4$ must be $v, w, y$, and $z$, since $v, w$, and $z$ cannot have neighbors outside of $\{v, w, x, y, z\}$ and $wz$ is not a red edge. Hence, $yw$ must have been drawn before $yz$. 
%\item When coloring $yz$ red created a bad $C_4$, it means that $yw$ is drawn earlier than $yz$, since $v, w,$ and $z$ cannot have neighbors outside of $\{v, w, x, y, z\}$. Now, $xz$ must be colored red since whenever $x$ or $y$ has a neighbor outside of $\{v, w, x, y, z\}$, $R$ is good, which is a contradiction.
\item If coloring $yz$ red would have created a bad $K_{1,3}$, then there must have been two vertices $s, t$ such that $ys$ and $yt$ are red, and these are drawn earlier than $yz$. Whenever $w\in\{s,t\}$ or not, $R$ is good, which is a contradiction.
%\item When coloring $yz$ red created a bad $K_{1,3}$, it means that there are two vertices $s, t$ such that $ys$ and $yt$ are red, and these are drawn earlier than $yz$. Whenever $w\in\{s,t\}$ or not, $R$ is good, which is a contradiction.

Note that there are no red edges between $v$ and $x$ or between $v$ and $y$. %Painter cannot color $xv$ or $yv$ red even though Builder drew them.   
\end{itemize}
\end{itemize}

\paragraph{Subcase 1-3} 
We may assume that coloring $xz$ red creates a bad $K_{1,3}$, say $T_1$.
%with leaves $v_1, v_2$ and $v_3$. 
If $z$ is the center of $T_1$, then since there is no red edge between $z$ and $w$, $G$ has $X_2$ as a subgraph, which is a contradiction. 
%the three edges of $T_1$ are $zx, zz_1$ and $zz_2$ for some new vertices $z_1$ and $z_2$, there is $X_2$ in $G$ as a subgraph. 
Therefore, we may assume that $x$ is a center of $T_1$, and $xz, xu_1$, and $xu_2$ are the three edges of $T_1$ with new vertices $u_1$ and $u_2$. 
By symmetry, we may assume that coloring $yz$ red creates a bad $K_{1,3}, say T_2$, with the center $y$.  
We may also assume that $yz, yv_1$, and $yv_2$ are the three edges of $T_2$ with new vertices $v_1$ and $v_2$. 
Note that $w$ is not necessarily distinct from $u_1, u_2, v_1, v_2$. 

\begin{itemize}
\item If $\abs{\{u_1,u_2\}\cap\{v_1,v_2\}}=0$, then it is easy to check that Painter can color one of $xz$ and $yz$ red since $T_1$ or $T_2$ must be good by property $A_3$, which is a contradiction.   

%\item If $\abs{\{u_1,u_2\}\cap\{v_1,v_2\}}=1$, then we may assume that it is $w$ and let $u_2=v_2=w$. 
\item If $\abs{\{u_1,u_2\}\cap\{v_1,v_2\}}=1$, then it is easy to check that Painter can color one of $xz$ and $yz$ red since $T_1$ or $T_2$ must be good by property $A_4$, which is a contradiction. 

\item 
%If $\abs{\{u_1,u_2\}\cap\{v_1,v_2\}}=2$, then we may assume that one of them is $w$. Let $u_1=v_1$ by $u$ and $u_2=v_2$ by $w$. 
If $\abs{\{u_1,u_2\}\cap\{v_1,v_2\}}=2$, then let $w_1=u_1=v_1$ and $w_2=u_2=v_2$. 
We may assume that $yz$ is drawn later than $xz$, by symmetry. 
Then right before Builder draws $yz$, each of $\{x,y,z,w_1,w_2\}$ cannot have neighbors outside of $\{x,y,z,w_1,w_2\}$, since otherwise $T_2$ becomes good when Painter colors $yz$ red. 
However, this is a contradiction since the red $C_4$ with vertices $x, w_1, y, w_2$, is bad in $G- e$.
This is because a red $C_4$ in a component of at most five vertices is always bad.  
\end{itemize}

\paragraph{Case 2} 
Assume a bad $K_{1, 3}$, say $S$, is created when Painter colors $e$ red, and without loss of generality let $x, y, u$, and $v$ be the vertices of $S$ so that $ux, xv$ are red edges. 
Now, $y$ cannot have neighbors outside of $\{x, z, u, v\}$ in $G$ since otherwise $S$ is good by property $A_2$ when $e$ is colored red in $G$. 
If there is a red edge between $u$ and $y$ and between $v$ and $y$, then this case is covered by Case $1$. 
Therefore, we may assume that there is no red edge between $u$ and $y$ and between $v$ and $y$ in $G$. 
Since the blue edge $yz$ must have two incident red edges in $G$, we may assume that the two red edges are incident with $z$, say $zs, zt$. 
We can check that $\{s,t\}=\{u,v\}$, since otherwise $S$ is good by property $A_4$ when $e$ is colored red in $G$. 
Since $G- e$ has a red $C_4$ with vertices $x, u, z, v$, say $R$, by the induction hypothesis, $R$ must be good. 
This implies that the component containing $R$ must have at least six vertices, thus, one of $x,y,z,u$, and $v$ has a neighbor outside of $\{x,y,z,u,v\}$. 
However, this implies that $S$ is good when $e$ is colored red, which is a contradiction. 

\paragraph{Case 3} 
Assume a bad $C_4$, say $R$, is created when Painter colors $e$ red, and let $xu, uv, vy$ be the red edges of $R$. 
Now each of $\{x,y,z,u,v\}$ cannot have neighbors outside of $\{x,y,z,u,v\}$, since otherwise $R$ is good in $G$. 
This also implies that there is no red $K_{1,3}$ in this component since $K_{1,3}$ in a component of at most five vertices must be bad. 
Then the blue edge $zx$ has no two red edges incident with it in $G$, which is a contradiction. 

Therefore, Strategy~\ref{strategy-M} works, and thus Painter wins the online Ramsey game for $C_3$ on $X_2$-free graphs.
\end{proof}

We present a winning strategy for Painter that can be used for the following proposition covering two cases.

\begin{strategy}\label{stgy-OP}
%When Builder draws an edge $e$, if a blue $C_3$ is made when Painter colors $e$ blue or there is no red edge incident to $e$, then Painter colors $e$ red. Otherwise, Painter colors $e$ blue.
When Builder draws an edge $e$, if a blue $C_3$ is made when Painter colors $e$ blue or there is no red edge incident to $e$, then Painter colors $e$ red. Otherwise, Painter colors $e$ blue.
\end{strategy}
%\begin{proposition}
%$C_3$ is avoidable on $K_{1,4}$-free graphs. 
%\end{proposition}
%\begin{proof}
%greedy without monochromatic $K_3, K_{1,3}$. 
%\end{proof}
\begin{proposition}\label{prop-O-free}
Let $X_3$ and $X_4$ be the graphs in Figure~\ref{fig-c3}.
Painter wins the online Ramsey game for $C_3$ on $X_3$-free graphs and
%Builder cannot force a monochromatic copy of $C_3$ on $X_3$-free graphs. 
Painter wins the online Ramsey game for $C_3$ on $X_4$-free graphs.
%Moreover, Builder cannot force a monochromatic copy of $C_3$ on $X_4$-free graphs.
\end{proposition}
\begin{proof}
We will prove both statements at the same time. 
Painter will use Strategy~\ref{stgy-OP}. 
We claim that Painter can always color the new edge $e=xy$ with Strategy~\ref{stgy-OP}. 
Let $G$ be the new graph when Builder draws $e$. 
We will use induction on the number of edges. 
The base case is trivial. 

By the induction hypothesis, we may assume that every blue edge is incident with at least one red edge and that there is no monochromatic $C_3$ in $G- e$.
The strategy fails when coloring $e$ blue and red results in a blue $C_3$ and a red $C_3$, respectively. Let $z_1, z_2$ be vertices such that $\{x,z_1,y\}$ and $\{x,z_2,y\}$ are vertices of the blue $C_3$ and the red $C_3$, respectively. 
We will prove that if the strategy fails, then $G$ has both $X_3$ and $X_4$ as subgraphs, which is a contradiction, and thus the strategy does not fail in either game. 

Without loss of generality, we may assume that Builder has drawn $xz_2$ later than $z_2y$. 
Consider the graph right after Builder drew $xz_2$. Note that $xz_2$ is incident to a red edge $z_2y$. 
Since Painter uses Strategy~\ref{stgy-OP} and Painter colored $xz_2$ red, there must be a blue $C_3$ when Painter colors $xz_2$ blue. 
Let $x, z_2, v$  be the vertices of the blue $C_3$. 
Note that $xv$ and $z_2v$ are drawn earlier than $xz_2$. 
If $v\neq z_1$, then $G$ contains both $X_3$ and $X_4$ as a subgraph, and thus $v$ must be the same as $z_1$. 

Now consider the graph right before Builder drew $xz_2$.
Since Builder has already drawn $xz_1$ and Painter colored it blue, $xz_1$ must have at least one incident red edge in $G$. 
This red edge is incident with either $x$ or $z_1$, but in both cases $G$ contains both $X_3$ and $X_4$ as a subgraph, which is a contradiction, and thus the strategy works. 

Therefore, Strategy~\ref{stgy-OP} works, and thus Painter wins the online Ramsey game for $C_3$ on both $X_3$-free graphs and $X_4$-free graphs.
%Painter color red initially. Color blue if there is incident red edge. Color red if blue $C_3$ is made by coloring blue.  
\end{proof}

%\begin{proposition}\label{prop-P-free}
%Builder cannot force a monochromatic copy of $C_3$ on $X_4$-free graphs. 
%\end{proposition}
%\begin{proof}
%Painter will use strategy~\ref{stgy-OP}.
%The strategy fails when coloring the new edge $e=xy$ with bluef red results in a blue $C_3$ and a red $C_3$. Let $z_1, z_2$ be vertices such that $xz_1y, xz_2y$ are blue and red paths, respectively. Without loss of generality, We may assume that $xz_2$  is drawn later than $z_2y$. Since Painter uses strategy~\ref{stgy-OP}, the reason that $xz_2$ is colored red is that blue $C_3$ is made by coloring $xz_2$ blue. Let $v$  be that vertex with blue $xv, vz_2$. If $v\neq z_1$ then the graph contains $X_4$ as a subgraph when Builder draws $e$. Therefore $v$ must be same with $z_1$. Then consider right before when $xz_2$ is drawn, since $xz_1$ is already drawn and colored blue, it must have a red edge incident to it. In either case that the red edge is incident to $x$ or $z_1$, the graph contains $X_4$ as a subgraph when Builder draws $e$, which means that the strategy works.
%%Painter color red initially. Color blue if there is incident red edge. Color red if blue $C_3$ is made by coloring blue.
%\end{proof}
%
%\begin{strategy}
%\end{strategy}
%
%\begin{proposition}
%Painter wins the online Ramsey game $(C_3,X_5$-free graphs$)$.
%%Builder cannot force a monochromatic copy of $C_3$ on $X_5$-free graphs. 
%\end{proposition}
%\begin{proof}
%Not proven yet.
%\end{proof}

\subsection{The final touch}\label{finaltouch}

In this subsection we prove Theorem~\ref{thm-c3}.
We need two additional lemmas to prove Theorem~\ref{thm-c3}.

\begin{figure}[h]
	\begin{center}
		\includegraphics[scale=0.6]{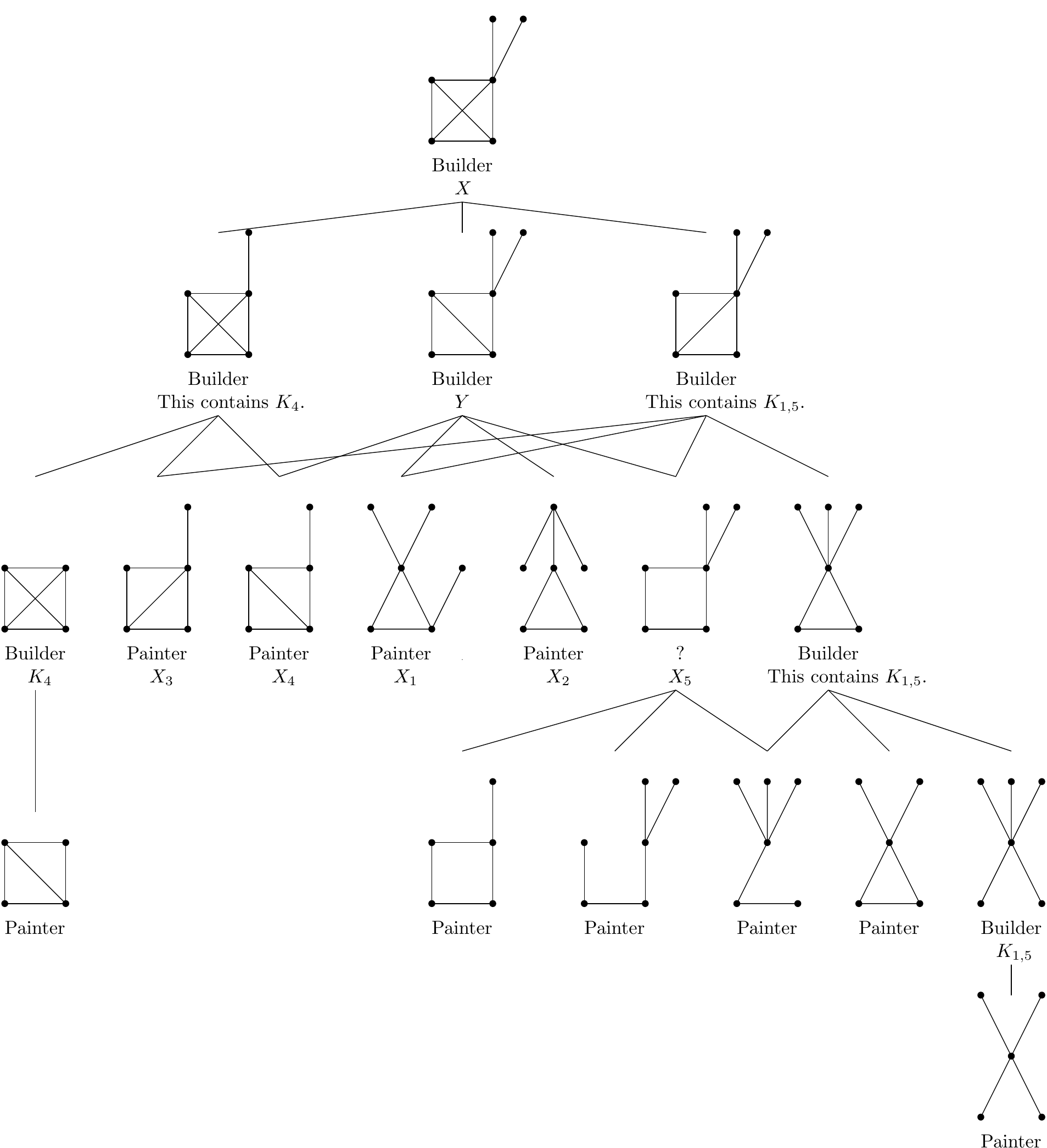}
	\end{center}
  \caption{The lines between graphs imply that the lower graph is a subgraph of the higher graph.}
  \label{fig-c3_1}
\end{figure}

\begin{lemma}\label{subgraphlemma_2}
If Builder wins the online Ramsey game for $H$ on $I$-free graphs for a graph $I$, then Builder wins the online Ramsey game for $H$ on $J$-free graphs for every graph $J$ that has $I$ as a subgraph. 
\end{lemma}
\begin{proof}
Since the set of $I$-free graphs is a subset of the set of $J$-free graphs, Builder can use the same strategy used in the case of $J$-free graphs.
\end{proof}
\begin{lemma}\label{subgraphlemma_1}
If Painter wins the online Ramsey game for $H$ on $I$-free graphs for a graph $I$, then Painter wins the online Ramsey game for $H$ on $J$-free graphs for every graph $J$ that is a subgraph of $I$. 
%If Builder cannot force a monochromatic copy of $C_3$ on $H$-free graphs for a graph $H$, then Builder cannot force a monochromatic copy of $C_3$ also on $I$-free graphs for every subgraph $I$ of $H$.
\end{lemma}
\begin{proof}
Since the set of $J$-free graphs is a subset of the set of $I$-free graphs, Painter can use the same strategy used in the case of $I$-free graphs.%by previous lemma, Builder wins the online Ramsey game $(C_3, H$-free graphs$)$, which is a contradiction to that Painter wins the online Ramsey game $(C_3,H$-free graphs$)$.
\end{proof}

Finally, we prove Theorem~\ref{thm-c3}.

\begin{proof}[Proof of Theorem~\ref{thm-c3}]

By Lemma~\ref{lemma-c3}, it is enough to consider when $F$ is a subgraph of $X$. 

By Propositions~\ref{prop-N-free},~\ref{prop-M-free}, and~\ref{prop-O-free}, along with Lemma~\ref{subgraphlemma_1}, Painter wins the online Ramsey game for $C_3$ on $F$-free graphs %Builder cannot force a monochromatic copy of $C_3$ on $H$-free graphs
if $F$ is isomorphic to a subgraph of a graph in $\{X_1,X_2,X_3,X_4\}$. 
By Propositions~\ref{prop-K_4-free}, \ref{prop-K_15-free}, and \ref{prop-A-free}, along with Lemma~\ref{subgraphlemma_2}, Builder wins the online Ramsey game for $C_3$ on $F$-free graphs %Builder cannot force a monochromatic copy of $C_3$ on $H$-free graphs
if $F$ contains a graph in $\{X_1,X_2,X_3,X_4\}$ as a proper subgraph.

It is easy to check that all graphs without isolated vertices are covered by the above paragraph except for the graph $X_5$. 
%For the only if part of the Theorem~\ref{thm-c3}, it is enough to show that if a graph $F$ is not isomorphic to a subgraph of $X_1,X_2,X_3$, or $X_4$, then Builder wins the online Ramsey game for $C_3$ on $F$-free graphs. 
%Builder can force a monochromatic copy of $C_3$ on $H$-free graphs.
%By Lemma~\ref{lemma-c3}, the case reduces to subgraphs of $X$.
Figure~\ref{fig-c3_1} shows subgraphs of $X$. 
Moreover, ``Builder'' and ``Painter'' written under some graph in Figure~\ref{fig-c3_1} means that Builder and Painter, respectively, wins the online Ramsey game for $C_3$ on $F$-free graphs. %We can easily see that $X_6,X_7,X_8,X_9,X_{10},X_{11}$ are isomorphic to one of subgraphs of $\{X_1,X_2,X_3,X_4\}$. 
%All the remaining cases are covered by Propositions~\ref{prop-K_4-free},~\ref{prop-K_15-free}, and~\ref{prop-A-free} with Lemma~\ref{subgraphlemma_2}.
%Note that if Builder cannot force a monochromatic copy of $C_3$ on $H$-free graphs, then Builder cannot force a monochromatic copy of $C_3$ also on $I$-free graphs for every $I$ which is subgraph of $H$. Also Note that if Builder can force a monochromatic copy of $C_3$ on $H$-free graphs, then Builder can force a monochromatic copy of $C_3$ also on $I$-free graphs for every $I$ which has $H$ as subgraph. By Lemma~\ref{lemma-c3}, for every graph $H$ not isomorphic to a subgraph of $X$, Builder can force a monochromatic copy of $C_3$ on $H$-free graphs. Then by Propositions ~\ref{prop-K_4-free},~\ref{prop-K_15-free},~\ref{prop-A-free},~\ref{prop-M-free},~\ref{prop-N-free} and ~\ref{prop-O-free}, Theorem~\ref{thm-c3} is proven supported by Figure~\ref{fig-c3_1} that these cover all the cases, since every proper subgraphs $J$ of $K_4$, $K_{1,5}$, $A$ or $X_5$ is a subgraph of at least one of $\{X_2, X_1, X_3, X_4\}$, which means that Builder cannot force $C_3$ on $J$-free graphs.
\end{proof} 

We end this section with the only case that is unsolved. 

\begin{question}
Let $X_5$ be the graph in Figure~\ref{fig-c3}. Who wins the online Ramsey game for $C_3$ on $X_5$-free graphs?
\end{question}

%\begin{theorem}
%$C_3$ is avoidable on $\{K_4-e+f\}$-free graphs where $f$ is a new edge that has a vertex of $K_4-e$ as an end.
%\end{theorem}
%\begin{proof}
%\end{proof}

%\begin{theorem}
%If $C_3$ is avoidable on $\{X\}$-free graphs, then $X\in\{K_4-e+2f, K_4-e+f\}$.
%\end{theorem}
%\begin{proof}
%TODO.
%\end{proof}

%\bibliographystyle{plain}
%\bibliography{online_ramsey}

\section*{Acknowledgments}

We thank the anonymous referee for helping us improve the readability of the paper.

\end{document}